\documentclass[12pt]{amsart}
\usepackage{amsmath,amssymb,mathtools}
\usepackage{microtype}
\usepackage[margin=1.08in]{geometry}
\usepackage{enumitem}
\usepackage{booktabs,longtable,array}
\usepackage{xcolor}
\usepackage[colorlinks=true,linkcolor=blue!55!black,citecolor=blue!55!black,urlcolor=blue!55!black]{hyperref}
\usepackage{fancyhdr}
\usepackage{esint}

\numberwithin{equation}{section}
\newtheorem{theorem}{Theorem}[section]
\newtheorem{proposition}[theorem]{Proposition}
\newtheorem{lemma}[theorem]{Lemma}
\newtheorem{conjecture}[theorem]{Conjecture}
\newtheorem{corollary}[theorem]{Corollary}
\newtheorem{remark}[theorem]{Remark}

\newcommand{\II}{\mathcal I}
\newcommand{\Sph}{\mathbb S}
\newcommand{\dd}{\,d}
\newcommand{\R}{\mathbb{R}}
\newcommand{\BB}{\mathcal{B}}
\newcommand{\HH}{\mathcal{H}}
\newcommand{\RR}{\mathcal{R}}
\newcommand{\KK}{\mathcal{K}}
\newcommand{\NN}{\mathcal{N}}
\newcommand{\QQ}{\mathcal{Q}}
\newcommand{\PP}{\mathcal{P}}
\newcommand{\VC}{\mathcal{V}}
\newcommand{\EE}{\mathcal{E}}
\newcommand{\eps}{\varepsilon}
\newcommand{\re}{\operatorname{Re}}
\newcommand{\D}{\mathbb{D}}
\newcommand{\C}{\mathbb{C}}
\newcommand{\im}{\operatorname{Im}}
\newcommand{\dist}{\operatorname{dist}}

\newcommand{\ind}{\mathbf 1}

\newcommand{\VV}{\mathrm{Vol}}
\newcommand{\diam}{\mathrm{diam}}

\author{Zhehui Wang}
\address{Zhehui Wang, School of Sciences, Great Bay University, Dongguan 523000, China}
\email{wangzhehui@gbu.edu.cn}

\title[The volumes of nodal and positivity sets of harmonic functions]{Sharp lower bounds for the volumes of nodal and positivity sets of harmonic functions}
\date{}

\begin{document}
\begin{abstract}
Let $B=B(p, 1)\subset\R^n$ be a unit ball and $n\ge 3$. We prove that there are positive constants $c$ and $C$, depending only on $n$, such that every non-zero real-valued harmonic function $u: 4B\to \R$ with $u(p)=0$ satisfies
$$
    \HH^{n-1}\bigl(\{u=0\}\cap 2B\bigr)\ge C\NN,
$$
and 
$$
  \HH^n\bigl(\{u>0\}\cap  {\frac 12}B\bigr)
  \ge c\bigl(\log(1+\NN)\bigr)^{1-n},
$$
where $\NN$ is the doubling index defined by $$\NN=\log_2\frac{\sup_{B}|u|}{\sup_{\frac{1}{2}B}|u|}.$$ The first estimate confirms a folklore conjecture on the nodal volume of harmonic functions, and its linear dependence on $\NN$ is optimal. The second estimate extends the planar result of Nazarov, Polterovich, and Sodin to higher dimensions, and the logarithmic order is optimal. As a further consequence of the ideas developed in the proof, we obtain an alternative proof of Nadirashvili's conjecture that does not rely on the multiscale analysis.
\end{abstract}
\maketitle

\section{Introduction}\label{sec:intro}
The size of a nodal set measures the oscillation of a solution to an elliptic equation, while the doubling index characterizes its growth rate. Establishing a connection between these two quantities is an important problem in nodal geometry. 

For Laplace eigenfunctions 
$-\Delta_g\varphi_\lambda=\lambda\varphi_\lambda$ on an $n$-dimensional closed smooth Riemannian manifold $(M,g)$, Yau's conjecture \cite{Yau1982} predicts
the two-sided estimate
$$
  c_{g}\sqrt{\lambda}
  \le
  \HH^{n-1}\bigl(\{\varphi_\lambda=0\}\bigr)
  \le
  C_{g}\sqrt{\lambda},
$$
where $c_{g},C_{g}>0$ depend only on the Riemannian metric on $M$ and are independent of the eigenvalue $\lambda$. Donnelly and Fefferman \cite{DonnellyFefferman1988} proved the conjectured two-sided estimate for real-analytic metrics. In the smooth category, the sharp lower
bound in dimension two goes back to Br{\"u}ning \cite{Bruning1978}. Important lower bounds in general dimensions were subsequently obtained by Colding and Minicozzi \cite{ColdingMinicozzi2011} (also see \cite{Steinerberger2014} for a proof using the heat flow), and Sogge and Zelditch \cite{SoggeZelditch2011, SoggeZelditch2012}. Later, Logunov \cite{Logunov2018} proved the conjectured lower bound in all dimensions. The upper bound direction of Yau's conjecture remains open; see \cite{Dong1992, DonnellyFefferman1990, HardtSimon1989, Logunov2018-2, Nadirashvili1988} and the references therein for important progress.

At the local level, the structure and size of nodal sets of solutions to elliptic equations were studied by Hardt and Simon \cite{HardtSimon1989}. Quantitative connections among vanishing order, frequency, growth, and nodal measure were developed in the work of Garofalo and Lin \cite{GarofaloLin1986} and Lin \cite{Lin1991}; see also Han's account \cite{Han} of nodal sets of harmonic functions. Connections between local growth and nodal geometry appear in the work of Nazarov, Polterovich and Sodin \cite{NazarovPolterovichSodin2005} on the local distribution of signs of eigenfunctions on surfaces. Roy-Fortin \cite{RoyFortin2015} established two-sided estimates for nodal length on closed surfaces in terms of the average local growth of eigenfunctions at the wavelength scale. Quantitative propagation of smallness provides a complementary connection between growth and the size of sublevel sets. A broader account of these developments and their relations to Yau's conjecture can be found in \cite{LogunovMalinnikovaReview2019}. More related results can be found in \cite{HardtEtAl1999, Kukavica1998, Mangoubi2013, NaberValtorta2017} and the references therein.

The present paper studies quantitative local questions concerning the volumes of nodal and positivity sets of harmonic functions in the Euclidean setting, with growth measured by the doubling index rather than by an eigenvalue. No result for general variable-coefficient equations is asserted here. Recall that Logunov's proof \cite{Logunov2018} of the lower bound direction of Yau's conjecture proceeds through a local theorem for harmonic functions, resolving a conjecture of Nadirashvili \cite{Nadi} and yielding a uniform local lower bound for the nodal volume of harmonic functions vanishing at the center of a ball \cite[Theorem 1.2]{Logunov2018}. For a uniform lower bound of nodal volume in the setting of elliptic homogenization, see Li and Ying \cite{LiYing-1}, and the related paper by Li, Wang and Ying \cite{LiWangYing}. Our aim is twofold. First, we strengthen this uniform local lower bound by providing a quantitative relation between the nodal volume and the growth of the harmonic function. Second, we establish a lower bound of optimal logarithmic order for the volume of the positivity set, extending the planar result of Nazarov, Polterovich and Sodin  \cite[Theorem 2.2]{NazarovPolterovichSodin2005} to higher dimensions and answering the question
raised in \cite[Section 7.4]{NazarovPolterovichSodin2005}. 

For $x\in \R^n$ and $r>0$, write the Euclidean ball $B=B(x ,r)$, and $kB=B(x ,kr)$ for a positive constant $k$. For any non-zero harmonic function $u$ defined in $2B$, the doubling index of $u$ in $B$ is defined by
$$\NN_u(B)=\log_2\frac{\sup_{2B}|u|}{\sup_{B}|u|}.$$
If no confusion arises, we drop the subscript $u$ and simply write $\NN(B)$. When the ball is centered at the center of homogeneity, the doubling index of a homogeneous harmonic polynomial is equal to its degree, so $\NN(B)$ can be regarded as a scale dependent analogue of degree.

\subsection{The volume of the nodal set}
In \cite{LPS}, Logunov,  Priya and Sartori established the following almost sharp lower bound for the nodal volume of harmonic functions.

\begin{theorem}[{\cite[Theorem 1.2]{LPS}}]\label{thm:LPS}
Let $B=B(p, 1)\subset\R^n$ be a unit ball and $n\ge 3$. For every $\eps>0$, there exists a positive constant $c_{n, \eps}$ depending only on $n$ and $\eps$ such that for every non-zero harmonic function $u: 4B\to \R$ with $u(p)=0$, we have
$$\HH^{n-1}\bigl(\{u=0\}\cap 2B\bigr)\ge c_{n, \eps}\NN\bigl(\frac{1}{2}B\bigr)^{1-\eps}.$$
\end{theorem}
For a fixed $0<\eps<1$, this estimate recovers the
uniform local lower bound proved by Logunov and strengthens
it quantitatively for functions with large doubling index;
see also \cite[Claim~A.4]{LPS}. A folklore conjecture, recorded as Conjecture 1.3 in \cite{LPS}, predicts that the nodal volume admits a lower bound linear in the doubling index. 
\begin{conjecture}[{\cite[Conjecture 1.3]{LPS}}]\label{conj}
    Let $B=B(p, 1)\subset\R^n$ be a unit ball and $n\ge 3$. For every non-zero harmonic function $u: 4B\to \R$ with $u(p)=0$, the following lower bound holds:
\begin{equation}\label{conjstat}
\HH^{n-1}\bigl(\{u=0\}\cap 2B\bigr)\ge c_n\NN\bigl(\frac{1}{2}B\bigr),\end{equation}
for some positive constant $c_n$ depending only on $n$.
\end{conjecture}
On the upper bound side, the linear control of nodal volume by quantitative growth goes back to the work of Donnelly--Fefferman \cite{DonnellyFefferman1988} and Lin \cite{Lin1991}. Donnelly--Fefferman developed the analytic growth approach to nodal sets and obtained linear bounds in terms of a suitable growth parameter, while Lin proved the corresponding estimate for harmonic functions in terms of the frequency; see also Han \cite{Han} for a detailed account. Using the standard comparison between frequency and the doubling index on comparable scales, one obtains the following theorem.

\begin{theorem}[\cite{DonnellyFefferman1988, Lin1991, Han}]\label{ndupper}
    Let $B\subset\R^n$ be a unit ball and $n\ge 3$. There exists a positive constant $C_n$, depending only on $n$, such that
$$\HH^{n-1}\bigl(\{u=0\}\cap B\bigr)\le C_n\NN\bigl(B\bigr),$$
 for every non-zero harmonic function $u: 2B\to \R$.
\end{theorem}

Our first main result gives an affirmative answer to Conjecture \ref{conj}. 
\begin{theorem}\label{thm:main}
Let $B=B(p, 1)\subset\R^n$ be a unit ball and $n\ge 3$. There is a positive constant $C_n$, depending only on $n$, such that
every non-zero harmonic function $u$ in the open ball $4B$ satisfies
\begin{equation}\label{eq:main-additive-1}
 \NN\bigl(\frac12 B\bigr)\leq C_n\bigl(1+\HH^{n-1}\bigl(\{u=0\}\cap 2B\bigr)\bigr).
\end{equation}
If, in addition, $u(p)=0$, then
\begin{equation}\label{eq:main-linear-2}
 \HH^{n-1}\bigl(\{u=0\}\cap 2B\bigr)\geq c_n\NN\bigl(\frac12 B\bigr),
\end{equation}
for a positive constant $c_n$ depending only on $n$.
\end{theorem}

In dimension two, the analogous estimate follows from classical bounds relating the growth of harmonic functions to their sign changes on circles, going back to Gelfond \cite{Gelfond1934}; see
\cite[formula (1.7)]{NazarovPolterovichSodin2005}
and \cite[Lemma 3.2.1]{RoyFortin2015}.

\subsection*{Optimality of the linear dependence in Theorem \ref{thm:main}}
The model $\re((x_1+ix_2)^m)$ appears in
\cite[Example~2.7]{LPS}, where it is used to describe the
distribution of doubling indices. Its extension to $\mathbb R^n$,
independent of $x_3,\ldots,x_n$, also shows that the linear
dependence in our estimate is optimal. For each integer $m\geq1$,
let
$$
  u_m(x)=\re((x_1+ix_2)^m).
$$
Then $u_m$ is harmonic, $u_m(0)=0$, and
$$
  \sup_{B_r}|u_m|=r^m,
  \qquad
  \log_2\frac{\sup_{B_1}|u_m|}
                 {\sup_{ B_{\frac12}}|u_m|}=m.
$$
Indeed, $|u_m(x)|\leq |x|^m$, and the supremum is approached
as $x$ tends to $(r,0,\ldots,0)$ from inside the ball.
Writing $x_1=s\cos\theta$ and $x_2=s\sin\theta$, we have
$u_m(x)=s^m\cos(m\theta)$. Hence its zero set consists of
$m$ distinct hyperplanes containing
$\{x\in\R^n:x_1=x_2=0\}$.
Their pairwise intersections have zero
$\HH^{n-1}$-measure, so
$$
  \HH^{n-1}\bigl(\{x\in B_2:u_m(x)=0\}\bigr)
  =
  m\,\HH^{n-1}\bigl(\{x\in B_2:x_1=0\}\bigr).
$$
The factor multiplying $m$ is a positive constant depending
only on $n$. Thus the nodal volume grows exactly linearly
with the doubling index, proving the optimality of this
order of growth.

\subsection{The volume of the positivity set} Beyond the size of the nodal set, a natural question is the distribution of signs of a harmonic function. More precisely, if a non-zero harmonic function vanishes at the center of a ball,
how small can the volume of its positivity set be in terms of its
doubling index?
In dimension two, Nazarov, Polterovich and Sodin \cite{NazarovPolterovichSodin2005} established the optimal logarithmic order for the positivity area in terms
of the doubling index.

\begin{theorem}[{\cite[Theorem 2.2]{NazarovPolterovichSodin2005}}]\label{NPS-2D}
Let $u$ be a non-zero harmonic function in the unit disc $\D$ vanishing at the origin. Then
$$
  \HH^2\bigl(\{u>0\}\cap\D\bigr)
  \ge \frac{c}{\log\beta^*(\D,u)},
$$
where
$$
  \beta(\D,u)
  =\log\frac{\sup_{\D}|u|}{\sup_{\frac12\D}|u|},
  \qquad
  \beta^*(\D,u)=\max\{\beta(\D,u),3\},
$$
and $c$ is a positive constant.
\end{theorem}
Thus Theorem \ref{NPS-2D} gives a reciprocal-logarithmic lower bound for the area of the positivity set in terms of the doubling index. Nazarov, Polterovich, and Sodin also constructed examples showing
that the logarithmic order of this lower bound is optimal; see \cite[Theorem 6.1]{NazarovPolterovichSodin2005}. In \cite[Section 7.4]{NazarovPolterovichSodin2005}, the same authors explicitly asked for the optimal lower bound, in terms of the doubling index, for the volume of the positivity set of a harmonic function in higher dimensions.
For $n\ge3$ and a non-zero harmonic function $u$ in the unit ball $B\subset\R^n$ vanishing at its center, they established
\begin{equation}\label{powerlow}
      \HH^n\bigl(\{u>0\}\cap B\bigr)
  \ge c_n\beta(B,u)^{1-n},
\end{equation}
where $\beta(B,u)$ is defined as above, and $c_n$ is a positive constant depending only on $n$. They noted, however, that this estimate should be far from sharp.

Our second main result is the following higher-dimensional logarithmic estimate. 
\begin{theorem}\label{thm:main-posi}
 Let $B=B(p, 1)\subset\R^n$ be a unit ball and $n\ge 3$. There exists a positive constant $c_n$, depending only on $n$, such that every non-zero harmonic function $u: B\to \R$ with $u(p)=0$ satisfies
\begin{equation}\label{eq:main-posi}
\HH^{n}\bigl(\{u>0\}\cap \frac{1}{2}B\bigr)
\ge c_n\bigl(\log\bigl(1+\NN(\frac12 B)\bigr)\bigr)^{1-n}.
\end{equation}
\end{theorem}
Applying the same result to $-u$ gives the identical lower bound
for the volume of the negative set. Since
$\beta(B,u)=(\log2)\NN(\frac12 B)$, the preceding theorem improves the power type lower bound \eqref{powerlow} to a logarithmic lower bound in every dimension $n\ge3$.

\subsection*{Optimality of the logarithmic order in Theorem \ref{thm:main-posi}}
By adapting the planar construction of Nazarov, Polterovich and Sodin \cite[Section 6.1]{NazarovPolterovichSodin2005}
to higher dimensions, we obtain a family of harmonic functions showing that the order $\bigl(\log\bigl(1+\NN(\frac12 B)\bigr)\bigr)^{1-n}$ in Theorem \ref{thm:main-posi} is optimal up to dimensional constants. We start from the entire function used in their construction
and combine spherical averaging with a contour deformation
to obtain the required higher-dimensional examples. More precisely,
Proposition \ref{sv:prop:sharpness} constructs a family of non-zero harmonic functions $u_R$ in $\R^n$, with $u_R(0)=0$, such that
$$
  \NN_R
  :=\log_2\frac{\sup_{B_1}|u_R|}{\sup_{B_{\frac12}}|u_R|}
  \longrightarrow\infty\qquad \text{ as }R\to \infty,
$$
and
$$
  \HH^n\bigl(\{u_R>0\}\cap B_{\frac12}\bigr)
  \le C_n\bigl(\log(1+\NN_R)\bigr)^{1-n},
$$
where $C_n$ is a positive constant depending only on $n$. These examples show that the logarithmic order of the lower bound in Theorem \ref{thm:main-posi} is optimal up to dimensional constants; see Remark \ref{sharprem}.

\subsection*{Outline of the proofs.} We use the notion of stable growth introduced by Logunov, Priya and Sartori \cite[Definition 6.1]{LPS}. We first establish the estimate under the stable growth assumption. The main new ingredient is a mechanism relating the doubling index and nodal volume through a fractional moment of the logarithmic gradient. Using Logunov's uniform local lower bound for nodal volume \cite[Theorem 1.2]{Logunov2018}, we bound this moment above in terms of the nodal volume. Conversely, our compactness argument for
normalized logarithms of holomorphic radial restrictions,
combined with a classical Remez type inequality for trigonometric
polynomials \cite[Theorem 2]{Erdelyi1992Remez}, gives a lower bound in terms of the doubling index. Comparing these bounds controls the doubling index by one
plus the nodal volume on stable growth balls. Then, we remove the stable growth assumption using the stable ball construction of \cite[Lemma 6.5]{LPS} and following the frequency layer decomposition and weighted
H\"older summation in \cite[Section 7, proof of Proposition 4.5]{LPS}. Finally, when $u(0)=0$, Logunov's uniform local lower bound absorbs the additive constant and yields the stated linear lower bound in Theorem \ref{thm:main}.

For Theorem \ref{thm:main-posi}, the additional ingredient we require is the
analytic small cube estimate of Donnelly and Fefferman \cite[Proposition 5.11]{DonnellyFefferman1988}, as formulated
in \cite[Proposition C.1]{LPS}. Together with an overlapping grid argument, it provides bounded local doubling outside a small exceptional set. Combining this with the fractional moment estimates, we select small balls centered at zeros and use the mean value property to show that the positivity set occupies a uniformly positive proportion of each ball. A disjoint ball selection then yields the corresponding lower bound on stable growth balls with sufficiently large
doubling index. For the general case, we reuse the frequency layer construction above, leaving at most $C_n\log\bigl(1+\NN(\frac12 B)\bigr)$ relevant layers. Unlike the nodal volume estimate, the local volume contributions carry no frequency weight, so the unweighted H\"older summation introduces a factor $m^{1-n}$, where $m$ is an upper bound for the number of retained layers. This yields the required logarithmic lower bound in Theorem \ref{thm:main-posi}.

As a further consequence of the ideas developed in the proof, we obtain an alternative proof of Nadirashvili's conjecture that avoids frequency layer decomposition and multiscale induction; see Appendix \ref{sec:appendix}. Our proof combines fractional moment estimates for the logarithmic gradient with an almost minimal nodal density argument. The final extremal step is similar in spirit to that of Logunov \cite[Section 7]{Logunov2018}, with the fractional moment estimates replacing the multiscale combinatorial input. We use the stable ball construction of Logunov, Priya and Sartori \cite[Lemma 6.5]{LPS}, but do not invoke the nodal volume lower bounds established in \cite{Logunov2018,LPS}.

\subsection*{Organization of the paper}This paper is organized as follows. In Section \ref{sec:compactness}, we establish a compactness principle for normalized logarithms of holomorphic
functions. In Section \ref{sec:gradient}, we derive upper bounds for fractional moments of the logarithmic gradient in terms of the nodal volume. In Section \ref{sec:stable}, we prove the linear lower bound under the stable growth assumption. In Section \ref{sec:main-proof}, we remove this assumption and complete the proof of Theorem \ref{thm:main}. In Section \ref{sec:main-posi-proof}, we establish the lower bound for the volume of the positivity set in Theorem \ref{thm:main-posi}, and show that its logarithmic order is optimal up to dimensional constants. Finally, in Appendix \ref{sec:appendix}, we give an alternative proof of Nadirashvili's conjecture that avoids the multiscale analysis.

\subsection*{Notation}
In the rest of this paper, we write $B_r=B(0, r)\subset\R^n$ for the Euclidean ball centered at the origin with radius $r$. All harmonic functions are real-valued. For a non-zero harmonic function
$u$ in $B_4$, set
\begin{equation*}\label{eq:basic-notation}
 M(r)=\sup_{B_r}|u|,
\end{equation*}
and define
\begin{equation*}\label{eq:N}
 \NN:=\NN(B_{\frac 12})=\log_2\frac{M(1)}{M(\frac12)}.
\end{equation*}
Whenever $u$ is defined in a domain $\Omega$, we write $\{u=0\}\cap \Omega\}=\{x\in \Omega: u(x)=0\}$ and $\{u>0\}\cap \Omega\}=\{x\in \Omega: u(x)>0\}$. The notation $\log$ denotes the natural logarithm; other bases are indicated explicitly. We write
$\D_r=\{z\in\C:|z|<r\}$. Surface measure is denoted by $d\sigma$, and normalized surface
measure on $\Sph^{n-1}$ by
$d\mu=|\Sph^{n-1}|^{-1}{d\sigma}$. The volume of a measurable set $X\subset\R^n$ is denoted by $\VV(X)$, and we write $\omega_n=\VV(B_1)$. $c_{a,b,\cdots},C_{a, b,\cdots}$, $c(a,b,\cdots)$ and $C(a, b,\cdots)$ are constants depending only on $a, b, \cdots$, which may change from line to line. 

\subsection*{Acknowledgements}The author is grateful to Josef Greilhuber, Aleksandr Logunov and Mikhail Sodin for pointing out that our method for proving the nodal volume estimate also yields a lower bound for the volume of the positivity set and an alternative proof of Nadirashvili's conjecture avoiding multiscale analysis. The author also thanks them for their helpful comments.
\subsection*{Disclosure on AI assistance.} The author used AI-assisted tools, principally ChatGPT. The author wrote and verified all theorem statements and proofs, and he takes full responsibility for the contents of the paper.

\section{A compactness lemma for logarithmic derivatives}\label{sec:compactness}
We establish the following compactness lemma using standard tools from complex analysis. We include the proof, with particular attention to the pointwise upper limit estimate needed in Corollary \ref{cor:separation}. Jensen's formula, normal family compactness, and the reflection principle are used in their classical forms; see \cite{Conway}. The positive reflected zero measures and the pointwise upper limit conclusion are part of the argument below. The values of a logarithm at real zeros are understood to be $-\infty$, and its derivative is taken in the ordinary sense away from these zeros. 

\begin{lemma}\label{lem:compactness}
Let $0<s<\frac{1}{4}$, $I=(-s,s)$, and $0<q<1$. Let $f_j$ be holomorphic in $\D_2$ and real-valued on $(-2,2)$. Suppose that there are a constant $C>0$ and a sequence $\{T_j\}_{j\ge 1}$ with $T_j\to\infty$ as $j\to\infty$, such that
\begin{equation}\label{eq:hol-growth}
 \sup_{\D_2}|f_j|\leq e^{CT_j},\qquad |f_j(0)|\geq e^{-CT_j}.
\end{equation}
Assume that
\begin{equation}\label{eq:small-derivative}
 \int_I\left|\frac{f'_j(x)}{T_jf_j(x)}\right|^q\dd x\longrightarrow0 \quad\text{ as } j\to\infty.
\end{equation}
Then, after passing to a subsequence, there is a finite real constant $c$ such
that
\begin{equation}\label{eq:L1-constant}
 v_j:=T_j^{-1}\log|f_j|\longrightarrow c\quad\text{in }L^1(I).
\end{equation}
For the same subsequence and every fixed $x\in I$, one also has
\begin{equation}\label{eq:pointwise-upper}
 \limsup_{j\to\infty}v_j(x)\leq c.
\end{equation}
\end{lemma}

Before the lengthy proof, we sketch the main steps here. The proof begins with Jensen's formula and a Blaschke factorization, which express $v_j=T_j^{-1}\log|f_j|$ as the sum of the real part of a uniformly controlled holomorphic function and the logarithmic potential of a positive measure associated with the zeros. After reflecting conjugate pairs of zeros into the lower half-plane, compactness yields a subsequential limit of $v_j$ in $L^1(I)$. A Cayley transform then converts the $L^q$-smallness of the normalized logarithmic derivatives into the $L^1$-smallness of the boundary imaginary parts of bounded holomorphic functions. Harmonic measure estimates and Schwarz reflection show that the auxiliary function $F$ extends holomorphically across $I$, with
$\re F=0$ on $I$. The relation between $F$ and the gradient of the limiting potential yields a finite constant boundary trace $c$. Monotonicity arising from the support of the measure in the closed lower half-plane identifies this trace with the potential at every
point of $I$, establishing $v_j\to c$ in $L^1(I)$.
Finally, truncation of the logarithmic kernel and weak convergence of the measures give $\limsup_{j\to\infty}v_j(x)\leq c$ for every
fixed $x\in I$, along the same subsequence.

\begin{proof}[Proof of Lemma \ref{lem:compactness}]
We may assume $T_j\geq1$. The four steps retain both the $L^1$ limit and the value information at every fixed real point. All products and sums over zeros in this proof count multiplicity.

{\it Step 1. Factorization of the zeros.}
Choose $R\in(1,\frac{5}{4})$ and $R'\in(\frac32,\frac74)$ such that none of the functions $f_j$ has a zero on either $\partial\D_R$ or $\partial\D_{R'}$. Such choices are possible because only countably many
radii are excluded. Since $f_j(0)\ne 0$, Jensen's formula and \eqref{eq:hol-growth} give
$$
 \sum_{\substack{|\alpha|<R' \text{ with } f_j(\alpha)=0\\ \text{ counted with multiplicity}}}\log\frac{R'}{|\alpha|}
 =\frac1{2\pi}\int_0^{2\pi}\log|f_j(R'e^{it})|\dd t
   -\log|f_j(0)|\leq2CT_j.
$$
If $\zeta$ is a zero of $f_j$ in $\D_R$, then $\frac{R'}{|\zeta|}>\frac{6}{5}$. Consequently, the number $m_j$ of zeros in $\D_R$, counted with multiplicity, satisfies
\begin{equation}\label{eq:zero-count}
 m_j\log(\frac65)\leq2CT_j.
\end{equation}
Set
\begin{equation*}\label{eq:blaschke}
 \BB_j(z)=\prod_{\substack{|\alpha|<R \text{ with } f_j(\alpha)=0\\ \text{ counted with multiplicity}}}
       \frac{R(z-\alpha)}{R^2-\overline\alpha z},
 \qquad g_j(z)=\frac{f_j(z)}{\BB_j(z)}.
\end{equation*}
The function $\BB_j$ is a finite Blaschke product in $\D_R$. The apparent singularities of $g_j$ at the zeros are removable, and $g_j$ has no zero in $\D_R$. Since $|\BB_j|=1$ on $\partial\D_R$ and $|\BB_j(0)|\leq1$, the maximum principle gives
$$
 \sup_{\D_R}|g_j|\leq e^{CT_j},\qquad |g_j(0)|\geq e^{-CT_j}.
$$
Choose a holomorphic logarithm of $g_j$ whose imaginary part at the origin belongs to $[-\pi,\pi]$. The positive harmonic functions
$$
 C+1-\re\bigl(T_j^{-1}\log g_j\bigr)
$$
have uniformly bounded values at the origin. Harnack's inequality bounds them in $\D_{\frac{7}{8}}$. The harmonic gradient estimate and the Cauchy--Riemann equations, together with the normalization at the origin, then give a uniform bound for $T_j^{-1}\log g_j$ in $\D_{\frac{3}{4}}$.

Define
\begin{equation*}\label{eq:Hj}
 H_j(z)=T_j^{-1}\log g_j(z)
       +T_j^{-1}\sum_{\substack{|\alpha|<R \text{ with } f_j(\alpha)=0\\ \text{ counted with multiplicity}}}
       \bigl(\log R-\log(R^2-\overline\alpha z)\bigr).
\end{equation*}
In each term of the sum, the logarithm is chosen to be real at the origin. Then 
\begin{equation*}\label{lemvj}
v_j=T_j^{-1}\log|f_j(z)|=\re H_j(z)+T_j^{-1}\sum_{\substack{|\alpha|<R \text{ with } f_j(\alpha)=0\\ \text{ counted with multiplicity}}}\log|z-\alpha|.\end{equation*}
For $|z|\leq\frac{3}{4}$,
$$
 \re(R^2-\overline\alpha z)\geq R(R-|z|)>\frac14,
 \qquad |R^2-\overline\alpha z|\leq\frac52.
$$
Thus these factors lie in a fixed compact subset of the right half-plane,
and the indicated logarithmic branches are uniformly bounded there.
It follows from~\eqref{eq:zero-count} that for $|z|\le \frac{3}{4}$

$$\left|T_j^{-1}\sum_{\substack{|\alpha|<R \text{ with } f_j(\alpha)=0\\ \text{ counted with multiplicity}}}
       \bigl(\log R-\log(R^2-\overline\alpha z)\bigr)\right|\le C_1T_j^{-1}m_j\le C_2,$$
where $C_1$ and $C_2$ are positive constants independent of $j$. Together with the uniform bound for $T_j^{-1}\log g_j$ in $\D_{\frac{3}{4}}$, we have
$$ \sup_j\|H_j\|_{L^\infty(\D_{\frac{3}{4}})}<\infty.$$
Cauchy's estimate for derivatives then gives a uniform bound for $H_j'$ in $\D_{\frac{1}{2}}$, and hence 
\begin{equation*}\label{eq:H-bounds}
 \sup_j\left(\|H_j\|_{L^\infty(\D_{\frac12})}
             +\|H'_j\|_{L^\infty(\D_{\frac12})}\right)<\infty.
\end{equation*}

Because $f_j$ is real on the real axis, Schwarz reflection and the identity
theorem show that its zeros are invariant under conjugation. Replace each
nonreal conjugate pair by mass $\frac{2}{T_j}$ at the zero in the lower half-plane,
and give each real zero mass $\frac{1}{T_j}$. This defines a positive measure $\mu_j$
on
$$
 K=\{a\in\C:|a|\leq\frac54,\ \im a\leq0\},\qquad \mu_j(K)\leq \frac{2C}{\log\frac{6}{5}},
$$
and the mass bound follows from \eqref{eq:zero-count}. Conjugate zeros have equal logarithmic moduli on the real axis. Hence
\begin{equation}\label{eq:potential-decomposition}
 v_j(x)=\int_K\log|x-a|\dd\mu_j(a)+\re H_j(x),\qquad x\in I.
\end{equation}
At a real zero, this identity is understood in the extended real sense.

{\it Step 2. Compactness of the logarithmic potentials.}
Define the map $\KK: K\to L^1(I)$ by $\KK(a)(x)=\log|x-a|.$ Then $\KK$ is continuous. To verify the only delicate case, let $a_0\in K\cap\R$. Fix $0<\delta<\frac16$ and assume that $a\in K$ satisfies
$|a-a_0|<\delta$. Set $E_\delta=I\cap(a_0-2\delta,a_0+2\delta).$
We estimate the integral separately on $E_\delta$ and its
complement. For $\zeta\in\{a,a_0\}$ and almost every $x\in E_\delta$,
we have
$$
|x-\re\zeta|
\le |x-\zeta|<3\delta<1.
$$
Consequently,
\begin{align*}
\int_{E_\delta}\bigl|\log|x-\zeta|\bigr|\,dx
\le
\int_{E_\delta}
-\log|x-\re\zeta|\,dx
\le
\int_{-3\delta}^{3\delta}-\log|t|\,dt
=6\delta\bigl(1-\log(3\delta)\bigr).
\end{align*}
It follows that
$$
\int_{E_\delta}\bigl|\log|x-a|-\log|x-a_0|\bigr|\,dx
\le 12\delta\bigl(1-\log(3\delta)\bigr).
$$

On the other hand, for $x\in I\setminus E_\delta$, we obtain 
$$
\bigl|\log|x-a|-\log|x-a_0|\bigr|\le \frac{|a-a_0|}{\delta}.
$$
Therefore,
$$
\|\KK(a)-\KK(a_0)\|_{L^1(I)}
\le
12\delta\bigl(1-\log(3\delta)\bigr)
+\frac{|I|}{\delta}|a-a_0|.
$$
First letting $a\to a_0$ with $\delta$ fixed, and then
letting $\delta\to0$, yields
$$
\|\KK(a)-\KK(a_0)\|_{L^1(I)}\longrightarrow0.
$$
Thus the map $\KK$ is continuous at every
$a_0\in K\cap\mathbb R$.
Continuity away from the real axis follows from uniform convergence. It also gives a uniform $L^1(I)$ bound for the kernels.

Since $K$ is a compact metric space and $\sup_j\mu_j(K)<\infty$, weak-$*$ sequential compactness yields, after passing to a subsequence, $\mu_j\rightharpoonup\mu$. Here, $\mu$ is a finite positive Borel measure. On the other hand, $\{H_j\}$ is a locally uniformly bounded family of holomorphic functions in $\mathbb D_{\frac{1}{2}}$. By Montel's theorem, we may pass to a further subsequence, still indexed by $j$, such that
$$
H_j\longrightarrow H
\qquad\text{locally uniformly in }\mathbb D_{\frac{1}{2}},
$$
where $H$ is holomorphic in $\mathbb D_{\frac{1}{2}}$.

For completeness, fix $\eps_0>0$. Compactness of $K$ and
continuity of $\KK$ give points $a_1,\ldots,a_L\in K$ and a continuous
partition of unity $\chi_1,\ldots,\chi_L$ such that
$$
 \left\|\KK(a)-\sum_{\ell=1}^L\chi_\ell(a)\KK({a_\ell})\right\|_{L^1(I)}
 \leq\eps_0\qquad(a\in K).
$$
For any positive finite measure $\lambda$ on $K$, integration gives
$$
 \left\|\int_K \KK(a)\dd\lambda(a)
 -\sum_{\ell=1}^L\left(\int_K\chi_\ell\dd\lambda\right)\KK({a_\ell})
 \right\|_{L^1(I)}\leq\eps_0\lambda(K).
$$
Applying this to $\lambda=\mu_j$ and $\lambda=\mu$, we have
\begin{align*}
\left\|
\int_K \KK(a)\,d\mu_j(a)-\int_K \KK(a)\,d\mu(a)
\right\|_{L^1(I)}
\le{}&
\eps_0\bigl(\mu_j(K)+\mu(K)\bigr)\\
&+\sum_{\ell=1}^L
\left|
\int_K\chi_\ell\,d\mu_j
-\int_K\chi_\ell\,d\mu
\right|
\|\KK({a_\ell})\|_{L^1(I)}.
\end{align*}
Weak convergence handles the
finitely many coefficients, so we fix $\eps_0>0$ first and let $j\to\infty$. Then the uniform mass bound handles the errors, and 
we further let $\eps_0\to 0$. Thus the left hand side tends to $0$ as $j\to\infty.$ Therefore,
\begin{equation}\label{eq:potential-L1}
 v_j\longrightarrow v(x):=\int_K\log|x-a|\,d\mu(a)+\re H(x)
 \quad\text{in }L^1(I).
\end{equation}
At this stage, the displayed potential is finite almost everywhere, by
Fubini's theorem and the uniform $L^1$ bound for the kernels.

{\it Step 3. Boundary reflection of a bounded holomorphic function.}
We now use~\eqref{eq:small-derivative}. Choose
$$
 C_0>1+\sup_j\|H'_j\|_{L^\infty(\D_{\frac{1}{2}})},
$$
and, in $\D_{\frac{1}{2}}\cap\{\im z>0\}$, set
\begin{equation*}\label{eq:Fj}
 F_j(z)=iC_0-H'_j(z)-\int_K\frac{\,d\mu_j(a)}{z-a}.
\end{equation*}
The location and positivity of the measures give
$$
 \im\left(-\frac1{z-a}\right)
   =\frac{\im z-\im a}{|z-a|^2}\geq0,
 \qquad z\in\D_{\frac{1}{2}}\cap\{\im z>0\},\ a\in K.
$$
It follows that $\im F_j\geq1$. At a real point which is not a zero,
differentiating~\eqref{eq:potential-decomposition} gives
\begin{equation}\label{eq:real-Fj}
 \re F_j(x)=-v'_j(x)=-\frac{f'_j(x)}{T_jf_j(x)}.
\end{equation}
The same lower bound for the imaginary part holds at such boundary points.

On every compact subset of $\D_{\frac{1}{2}}\cap\{\im z>0\}$, the kernels $(z-a)^{-1}$ and their $z$-derivatives are uniformly bounded and continuous in $a\in K$, since
$|z-a|\geq\im z$. Weak convergence and a finite net argument therefore imply locally uniform convergence of the Cauchy transforms $\int_K\frac{\,d\mu_j(a)}{z-a}$. Hence $F_j$ converges
locally uniformly in $\D_{\frac{1}{2}}\cap\{\im z>0\}$ to the finite holomorphic function
\begin{equation}\label{eq:F-limit}
 F(z)=iC_0-H'(z)-\int_K\frac{\dd\mu(a)}{z-a}.
\end{equation}
In particular, $\im F\geq1$.

Consider the Cayley transforms
\begin{equation}\label{eq:cayley}
 w_j=\frac{F_j-i}{F_j+i},\qquad w=\frac{F-i}{F+i}.
\end{equation}
They satisfy $|w_j|\leq1$, and $w_j\to w$ locally uniformly in $\D_{\frac{1}{2}}\cap\{\im z>0\}$.
Moreover, $w$ cannot take the value $1$ in $\D_{\frac{1}{2}}\cap\{\im z>0\}$, since $F$ is finite there. For real $x$ away from the zeros, write $F_j(x)=F^{(1)}_j+iF^{(2)}_j$, where $F_j^{(2)}\geq1$. Direct calculation gives
$$\im w_j(x)=\frac{-2F_j^{(1)}}{(F_j^{(1)})^2+(F_j^{(2)}+1)^2}.$$
By \eqref{eq:real-Fj}, for $0<q<1$,
$$
 |\im w_j(x)|\leq\min\{1,2|v'_j(x)|\}
 \leq C_q|v'_j(x)|^q.
$$
Consequently, by the assumption \eqref{eq:small-derivative},
\begin{equation}\label{eq:cayley-trace}
 \|\im w_j\|_{L^1(I)}\longrightarrow0.
\end{equation}
At a real zero of multiplicity $m$, the function $F_j$ has a simple pole with
residue $-\frac{m}{T_j}$. The formula
$$
 w_j=1-\frac{2i}{F_j+i}
$$
shows that this pole becomes a removable singularity of $w_j$, with value
$1$. Thus each $w_j$ is continuous on $I$, and these isolated points do not
affect~\eqref{eq:cayley-trace}.

We claim that $w$ extends holomorphically across $I$ and is real-valued there. Fix an upper half disk $V$ whose closed diameter $J_0$ lies in $I$, and denote
its semicircular arc by $\Gamma$. Apply harmonic measure representation in
$V$ to $\im w_j$. For fixed $z\in V$, the harmonic measure on $J_0$ is
dominated by the Poisson measure of the upper half-plane. To see this
domination, extend nonnegative boundary data on $J_0$ by zero to the real
line. Its upper half-plane Poisson extension dominates the solution in $V$
with zero data on $\Gamma$, by the maximum principle. The density on $J_0$
is therefore at most $(\pi\im z)^{-1}$. The absolute value of the diameter
contribution is bounded by
$$
 \frac1{\pi\im z}\int_{J_0}|\im w_j(x)|\,d x,
$$
which tends to zero by~\eqref{eq:cayley-trace}. 
Let $\omega_V^z$ denote the harmonic measure in $V$. Since $|\im w_j|\leq1$, we have
$$|\im w_j(z)|\le \frac1{\pi\im z}\int_{J_0}|\im w_j(x)|\,d x+\omega_V^z(\Gamma).$$
By letting $j\to\infty$, we obtain
$$
 |\im w(z)|\leq\omega_V^z(\Gamma).
$$
The right-hand side tends to zero as $z$ approaches an interior point of
$J_0$ from $V$. Thus $\im w$ has a continuous zero trace on $I$.

Odd reflection extends $\im w$ harmonically across each compact subinterval
of $I$. Taking a harmonic conjugate on a small disk extends $w$
holomorphically: in the upper half-disk, the constructed function and $w$
have the same imaginary part, so they differ by a real constant, which we
adjust. The resulting extension satisfies
$$
 w(z)=\overline{w(\overline z)}\qquad\text{for }\im z<0,
$$
locally across the interval. Its modulus remains at most $1$ on both sides.
If the extended function took the value $1$ at a point of $I$, the maximum
modulus principle would force it to be identically $1$ in a neighborhood.
This contradicts~\eqref{eq:cayley} in $\D_{\frac{1}{2}}\cap\{\im z>0\}$. Therefore
$F=\frac{i(1+w)}{1-w}$ extends holomorphically across $I$ and satisfies
\begin{equation}\label{eq:F-realzero}
 \re F(x)=0,\qquad x\in I.
\end{equation}
The locally defined extensions agree on overlaps by the identity theorem.

{\it Step 4. Identification of the trace and the pointwise upper limit.}
For $z\in\D_{\frac{1}{2}}\cap\{\im z>0\}$, let
$$
 U(z)=\int_K\log|z-a|\dd\mu(a)+\re H(z).
$$
This is a finite harmonic function: on each compact subset of $\D_{\frac{1}{2}}\cap\{\im z>0\}$ the kernel and all its derivatives are bounded uniformly in $a\in K$.
By \eqref{eq:F-limit}, writing $z=x+iy$, we have
\begin{equation*}\label{eq:U-gradient}
 U_x=-\re F,\qquad U_y=\im F-C_0.
\end{equation*}
The extension of $F$ bounds $\nabla U$ in an upper neighborhood of every
compact subinterval $J_1\Subset I$. In particular, for sufficiently small
$y,y'>0$,
$$
 \sup_{x\in J_1}|U(x+iy)-U(x+iy')|\leq C_{J_1}|y-y'|.
$$
Thus $U(\cdot+iy)$ has a finite, locally uniform boundary limit. Its
horizontal Lipschitz bound makes this trace continuous. If
$[x_1,x_2]\Subset I$, then by \eqref{eq:F-realzero},
$$
 U(x_2+iy)-U(x_1+iy)
 =-\int_{x_1}^{x_2}\re F(t+iy)\dd t\longrightarrow0\qquad\text{ as } y\to 0.
$$
The trace is therefore a single finite constant $c$
on the connected interval $I$.

For every fixed $x\in I$, the functions $\log|x+iy-a|$, $a\in K$, decrease to $\log|x-a|$ as $y\to 0$. Here the half-plane support condition is essential: writing $a=\alpha+i\gamma$ with $\gamma\leq0$ gives
$$
 |x+iy-a|^2=(x-\alpha)^2+(y-\gamma)^2.
$$
The logarithms have a common finite upper bound for small positive $y$. Monotone convergence, applied to that upper bound minus the kernel, and continuity of $H$ yield
\begin{equation}\label{eq:pointwise-potential}
 \int_K\log|x-a|\, d\mu(a)+\re H(x)=c,\qquad x\in I.
\end{equation}
In particular, the potential is finite at every point of $I$, not merely almost everywhere. Together with \eqref{eq:potential-L1}, this proves \eqref{eq:L1-constant}.

Fix $x\in I$. For $L>0$, the truncated kernel
$$
 \KK_L(a)=\max\{\log|x-a|,-L\}
$$
is continuous and bounded on $K$, with the value $-L$ assigned at $a=x$. Positivity of the measures and weak convergence give
$$
 \limsup_j\int_K\log|x-a|\dd\mu_j(a)
 \leq\int_K \KK_L(a)\dd\mu(a).
$$
Let $L\to\infty$, using~\eqref{eq:pointwise-potential}, and add the convergent term $\re H_j(x)$. This proves \eqref{eq:pointwise-upper}. The argument applies to every fixed $x$ for the subsequence already selected; no further
point dependent subsequence is used. This completes the proof.
\end{proof}

\begin{corollary}\label{cor:separation}
Fix constants $C>0$, $b>0$, $0<\eps<b$, $0<q<1$, and set $I=(-s,s)$ with $0<s<\frac{1}{4}$.
Let $J\Subset I\setminus\{0\}$ be a closed interval of positive length.
There are constants $\kappa>0$ and $T_0>0$, depending only on these parameters,
with the following property: Let $f$ be holomorphic in $\D_2$ and real-valued on $(-2,2)$. Suppose that for $T>T_0$, $f$ satisfies
\begin{equation}\label{eq:separation-hyp}
 \sup_{\D_2}|f|\leq e^{CT},\qquad
 |f(0)|\geq e^{-\eps T},\qquad
 \sup_J|f|\leq e^{-bT},
\end{equation}
then
\begin{equation}\label{eq:separation-concl}
 \int_I\left|\frac{f'(x)}{Tf(x)}\right|^q\dd x\geq\kappa.
\end{equation}
\end{corollary}

\begin{proof}
We argue by contradiction. Suppose there are $T_j\geq j$ and functions $f_j$
satisfying~\eqref{eq:separation-hyp} for which the integral
in~\eqref{eq:separation-concl} tends to zero. Apply
Lemma~\ref{lem:compactness}, with the growth constant enlarged to
$\max\{C,\eps\}$. For a subsequence,
$T_j^{-1}\log|f_j|\to c$ in $L^1(I)$. By \eqref{eq:pointwise-upper} at the origin and \eqref{eq:separation-hyp},
\begin{equation}\label{eq:separation-contradiction}
 -\eps\leq\limsup_j T_j^{-1}\log|f_j(0)|\leq c,
 \qquad
 c|J|=\lim_j\int_JT_j^{-1}\log|f_j|\dd x\leq-b|J|.
\end{equation}
Since $|J|>0$, \eqref{eq:separation-contradiction} yields
$-\eps\leq c\leq-b$, which contradicts $\eps<b$. This completes the proof.
\end{proof}

\begin{remark}\label{rem:compactness}
The constants in Corollary~\ref{cor:separation} need not be explicit, and their uniform dependence on the fixed parameters is sufficient. The pointwise upper limit in Lemma \ref{lem:compactness} is essential: convergence in $L^1(I)$ alone does not preserve a lower bound at a prescribed point. The proof of Lemma \ref{lem:compactness} also does not invoke a Poincar\'e inequality with exponent $q<1$, and the conclusion follows instead from positivity of the reflected zero measures and the holomorphic boundary reflection argument.
\end{remark}

\section{Nodal volume and logarithmic gradients}\label{sec:gradient}
In this section, we first recall the Euclidean consequence of Logunov's theorem \cite[Theorem 1.2]{Logunov2018}. 
\begin{theorem}[Euclidean setting, \cite{Logunov2018}]\label{logunovlowerbound}
    There exists $c_n>0$ depending only on $n$, such that for any non-zero harmonic function $u$ in $B(p, r)\subset\R^n$ with $u(p)=0$, one has
    \begin{equation}\label{eq:nadirashvili}
 \HH^{n-1}\bigl(\{u=0\}\cap B(p,r)\bigr)\geq c_nr^{n-1}.
\end{equation}
\end{theorem}
The above result was conjectured by Nadirashvili, and proved by Logunov. The constant $c_n$ is uniform under Euclidean translations and dilations. The
argument below combines \eqref{eq:nadirashvili} with a covering calculation
and the classical gradient estimate for positive harmonic functions. In particular, it does not assume a lower bound involving $\NN$.

\begin{proposition}\label{prop:gradient-upper}
Let $u$ be a non-zero harmonic function in $B_4\subset\R^n$,
where $n\geq3$. For every $0<q<1$,
\begin{equation}\label{eq:gradient-upper}
 \int_{B_{\frac{8}{5}}}\bigl|\nabla\log|u|\bigr|^q\dd x\leq C_{n,q}\bigl(1+\HH^{n-1}\bigl(\{u=0\}\cap B_2\bigr)\bigr)^q,
\end{equation}
where $C_{n, q}$ is a positive constant depending only on $n$ and $q$.
\end{proposition}

\begin{proof}
We may assume $\HH^{n-1}\bigl(\{u=0\}\cap B_2\bigr)<\infty$. Note $\{u=0\}$ has zero Lebesgue measure. Define
\begin{equation*}
g(x) = \left\{
\begin{aligned}
\frac{|\nabla u(x)|}{|u(x)|} \quad & \quad \text{ if } x\in B_{\frac85}\setminus\{u=0\}, \\
0 \quad &\quad \text{ if } x\in \{u=0\}\cap B_{\frac{8}{5}}.
\end{aligned}
\right.
\end{equation*}

{\it Step 1. A zero-free ball controls the logarithmic gradient.}
Write $$d(x)=\dist(x, \{u=0\}\cap B_2),$$ with $d(x)=\infty$ if the zero set is
empty. For $x\in B_{\frac{8}{5}}\setminus \{u=0\}$, the ball
\begin{equation}\label{eq:zero-free-ball}
 B\!\left(x,\min\left\{\frac1{20},\frac{d(x)}2\right\}\right)
 \subset B_2\setminus \{u=0\}
\end{equation}
is connected, so either $u$ or $-u$ is positive there. If $v>0$ is harmonic
in $B(x,R)$, the gradient estimate gives
\begin{equation}\label{eq:positive-gradient}
 |\nabla v(x)|\leq\frac nR v(x).
\end{equation}
Applying \eqref{eq:positive-gradient} in
\eqref{eq:zero-free-ball}, we obtain
\begin{equation}\label{eq:gradient-distance}
 g(x)\leq\frac{C_n}{\min\{\frac{1}{20},d(x)\}}.
\end{equation}
The fixed scale $\frac{1}{20}$ is necessary even when $ \{u=0\}\cap B_2$ is empty.

{\it Step 2. The volume of the thin neighborhood of the zero set.}
Fix $0<t<\frac{1}{20}$. Choose a maximal $t$-separated set
$\{z_1,\ldots,z_m\}\subset \{u=0\}\cap B_{\frac{9}{5}}$. It is finite by Euclidean
packing, and
\begin{equation}\label{eq:packing-cover}
 B(z_i,\frac{t}{3})\cap B(z_j,\frac{t}{3})=\varnothing\quad(i\ne j),\qquad
  \{u=0\}\cap B_{\frac{9}{5}}\subset\bigcup_{i=1}^m\overline{B(z_i,t)}.
\end{equation}
All closed balls $\overline{B(z_i,\frac{t}{3})}$ lie in $B_2$. Hence
\eqref{eq:nadirashvili} and the disjointness in \eqref{eq:packing-cover} give
$$ mc_nt^{n-1}
 \leq\sum_{i=1}^m\HH^{n-1}( \{u=0\}\cap B(z_i,\frac{t}{3}))
 \leq \HH^{n-1}\bigl(\{u=0\}\cap B_2\bigr),$$
and hence
\begin{equation}\label{eq:packing-count}
m\leq C_nt^{1-n}\HH^{n-1}\bigl(\{u=0\}\cap B_2\bigr).
\end{equation}
If $x\in B_{\frac{8}{5}}$ and $d(x)<t$, choose a zero $z$ with $|x-z|<t$. Then
\begin{equation}\label{eq:tube-cover}
 |z|<\frac85+\frac1{20}<\frac95,
 \qquad |z-z_i|\leq t\ \text{for some }i,
 \qquad |x-z_i|<2t.
\end{equation}
It follows from \eqref{eq:packing-count} and \eqref{eq:tube-cover} that
\begin{equation}\label{eq:tube-estimate}
 \bigl|\{x\in B_{\frac{8}{5}}:d(x)<t\}\bigr|
 \leq m\omega_n(2t)^n\leq C_nt\HH^{n-1}\bigl(\{u=0\}\cap B_2\bigr).
\end{equation}
If the set of centers is empty (i.e., $m=0$), then the set on the left hand side of \eqref{eq:tube-estimate} is empty as well. This proof uses no smooth tubular
parameterization of $\{u=0\}$ and therefore includes its singular points.

{\it Step 3. Distribution estimate and integration.}
Choose $\lambda_0(n)$ so that $\frac{C_n}{\lambda_0}<\frac{1}{20}$, with the constant $C_n$ from
\eqref{eq:gradient-distance}. For $\lambda\geq\lambda_0$, it gives
\begin{equation*}\label{eq:tail-inclusion}
 \{x\in B_{\frac{8}{5}}:g(x)>\lambda\}
 \subset\{x\in B_{\frac{8}{5}}:d(x)<\frac{C_n}{\lambda}\}.
\end{equation*}
Thus \eqref{eq:tube-estimate} implies for $\lambda\ge\lambda_0$
\begin{equation}\label{eq:weak-gradient}
 |\{x\in B_{\frac{8}{5}}:g(x)>\lambda\}|
 \leq\min\left\{|B_{\frac{8}{5}}|,C_n\lambda^{-1}\HH^{n-1}\bigl(\{u=0\}\cap B_2\bigr)\right\}.
\end{equation}
Set $\lambda_*=C_n\left(1+\HH^{n-1}\bigl(\{u=0\}\cap B_2\bigr)\right)\geq\lambda_0$. Since $g$ is nonnegative, using \eqref{eq:weak-gradient}, we get
\begin{align*}
 \int_{B_{\frac{8}{5}}}g^q\,d x
 &=q\int_0^\infty\lambda^{q-1}
       |\{x\in B_{\frac{8}{5}}:g(x)>\lambda\}|\dd\lambda\notag\\
 &\leq |B_{\frac{8}{5}}|\lambda_*^q
       +C_nq\HH^{n-1}\bigl(\{u=0\}\cap B_2\bigr)\int_{\lambda_*}^\infty\lambda^{q-2}\,d\lambda\notag\\
 &=|B_{\frac{8}{5}}|\lambda_*^q
       +\frac{C_nq}{1-q}\lambda_*^{q-1}\HH^{n-1}\bigl(\{u=0\}\cap B_2\bigr)\notag\\
 &\leq C_{n,q}\bigl(1+\HH^{n-1}\bigl(\{u=0\}\cap B_2\bigr)\bigr)^q.\label{eq:layer-cake}
\end{align*}
This completes the proof.
\end{proof}

\section{Harmonic functions with stable growth}\label{sec:stable}
Following the stable growth formulation introduced by Logunov, Priya and Sartori in \cite[Definition 6.1]{LPS}, we consider the condition
\begin{equation}\label{eq:stable}
 \NN(B_2)\le S\NN(B_{\frac{1}{2}}),   
\end{equation}
where $S\ge 1$ is fixed. We say that a non-zero harmonic function $u: B_4\to \R$ has {\it $S$-stable growth} in $B_1$ if its doubling index satisfies \eqref{eq:stable}. Throughout this section, we denote $\NN=\NN(B_{\frac{1}{2}})$.

The following is the main result in this section, and the assumption of $S$-stable growth will be removed in Section \ref{sec:main-proof}.
\begin{proposition}\label{prop:stable}
For every $n\geq3$ and $S\geq1$, there is $C_{n,S}>0$ such that every non-zero harmonic function $u:B_4\to\R$ satisfying the $S$-stable growth \eqref{eq:stable} has the following estimate.
\begin{equation}\label{eq:stable-conclusion}
 \NN\leq C_{n,S}\bigl(1+\HH^{n-1}\bigl(\{u=0\}\cap B_2\bigr)\bigr).
\end{equation}
\end{proposition} 

We first record some facts used in the proof; see
\cite[Chapter 5]{ABR} and \cite[Section 3 and Appendix A]{LPS}.
For $0<r<4$, let
\begin{equation}\label{eq:spherical-mean}
 L(r)^2=\int_{\Sph^{n-1}}u(r\omega)^2\dd\mu(\omega),
 \qquad d\mu=\frac{d\sigma}{|\Sph^{n-1}|}.
\end{equation}
Write $u=\sum_{k\geq0}P_k$ as its expansion in homogeneous harmonic
polynomials, then orthogonality of different degrees on the sphere gives
\begin{equation*}\label{eq:spherical-expansion}
 L(r)^2=\sum_{k\geq0}a_kr^{2k},\qquad
 a_k=\int_{\Sph^{n-1}}P_k(\omega)^2\dd\mu(\omega)\geq0.
\end{equation*}
Since $u$ is a non-zero harmonic function, at least one $a_k$ is positive. Hence $L(r)>0$, and
\begin{equation*}\label{eq:frequency}
 \beta(r)=\frac{d\log L(r)}{d\log r}
 =\frac{rL'(r)}{L(r)}
 =\frac{\sum_{k\geq0}ka_kr^{2k}}{\sum_{k\geq0}a_kr^{2k}}.
\end{equation*}
To display monotonicity, put $$p_k(r)=\frac{a_kr^{2k}}{\sum_j a_jr^{2j}}$$ in this
calculation only. Then
\begin{equation}\label{eq:frequency-weights}
 \sum_kp_k=1,\qquad \beta=\sum_k kp_k,\qquad
 \frac{dp_k}{d\log r}=2(k-\beta)p_k.
\end{equation}
Differentiation of
\eqref{eq:frequency-weights} gives
\begin{align}
 \frac{d\beta}{d\log r}
 &=2\sum_k k(k-\beta)p_k
 =2\left(\sum_k k^2p_k-\beta^2\right)\notag\\
 &=2\sum_k(k-\beta)^2p_k\geq0.
 \label{eq:frequency-monotonicity}
\end{align}
All termwise differentiations are justified on compact subintervals of
$(0,4)$. In particular, the function $t\mapsto\log L(e^t)$ is convex.
This is convexity with respect to $\log r$, not to $r$. Equivalently, if 
$$
 \log r_2=\theta\log r_1+(1-\theta)\log r_3,
 \qquad 0<r_1<r_2<r_3<4,
$$
then
\begin{equation}\label{eq:three-sphere}
 L(r_2)\leq L(r_1)^\theta L(r_3)^{1-\theta}.
 \end{equation}
Thus the intermediate radius is the geometric interpolation
$r_2=r_1^\theta r_3^{1-\theta}$.

Differentiating \eqref{eq:spherical-mean} and applying Green's identity yields
\begin{equation*}\label{eq:almgren-identity}
 2L(r)L'(r)
 =\frac{2r^{1-n}}{|\Sph^{n-1}|}
       \int_{\partial B_r}u\partial_\nu u\,d\sigma
 =\frac{2r^{1-n}}{|\Sph^{n-1}|}
       \int_{B_r}|\nabla u|^2\dd x,
\end{equation*}
and hence
\begin{equation*}\label{eq:almgren-identity-2}
 \beta(r)=\frac{r\int_{B_r}|\nabla u|^2\,d x}
                  {\int_{\partial B_r}u^2\,d\sigma}.
\end{equation*}
This is the usual Almgren frequency. Integrating
$\frac{L'(r)}{L(r)}=\frac{\beta(r)}{r}$ gives the identity used repeatedly below:
\begin{equation}\label{eq:frequency-ratio}
 \log\frac{L(r_2)}{L(r_1)}
 =\int_{r_1}^{r_2}\frac{\beta(t)}t\dd t,
 \qquad 0<r_1<r_2<4.
\end{equation}

We also use a direct comparison between the maximum and the spherical
mean. In normalized surface measure the Poisson formula reads
\begin{equation*}\label{eq:poisson-kernel}
 u(x)=\int_{\Sph^{n-1}}
 \frac{1-|\frac{x}{R}|^2}{|\frac{x}{R}-\omega|^n}u(R\omega)\dd\mu(\omega),
 \qquad |x|<R<4.
\end{equation*}
For $0<a<R$, the kernel has bounded $L^2(d\mu)$ norm uniformly for $|x|\leq a$. Recalling that $M(a)=\sup_{B_a}|u|$, the Cauchy--Schwarz inequality gives
\begin{equation}\label{eq:poisson-comparison}
 L(a)\leq M(a),\qquad M(a)\leq C\bigl(n,\frac a R\bigr)L(R).
\end{equation}
All radius ratios used
in Lemmas~\ref{lem:inner-decay} and \ref{lem:outer-growth} are fixed, so their
comparison constants depend only on $n$.

\begin{lemma}\label{lem:inner-decay}
For every non-zero harmonic function $u: B_4\to \R$, we have
\begin{equation}\label{eq:frequency-lower}
 \beta\bigl(\frac{21}{20}\bigr)\geq c_n\NN-C_n,
\end{equation}
where $c_n$ and $C_n$ are positive constants depending only on $n$.
Moreover, there exist positive constants $b$ and $N_1$ depending only on $n$, such that
\begin{equation}\label{eq:inner-decay}
 M\bigl(\frac{149}{100}\bigr)\leq e^{-b\NN}M\bigl(\frac{3}{2}\bigr)\qquad \text{ if } \NN\geq N_1.
\end{equation}
\end{lemma}

\begin{proof}
Applying \eqref{eq:poisson-comparison} at radii $1<\frac{21}{20}$ and at $\frac{1}{2}$, we have
$$ M(1)\leq C_0L\bigl(\frac{21}{20}\bigr),\qquad L\bigl(\frac{1}{2}\bigr)\leq M\bigl(\frac{1}{2}\bigr),$$
where $C_0\geq1$ depends only on $n$. Therefore,
\begin{equation*}\label{eq:inner-L-ratio}
 \log\frac{L\bigl(\frac{21}{20}\bigr)}{L\bigl(\frac{1}{2}\bigr)}\geq \NN\log2-\log C_0,
\end{equation*}
By
\eqref{eq:frequency-monotonicity} and \eqref{eq:frequency-ratio},
\begin{equation*}\label{eq:inner-frequency-endpoint}
 \NN\log2-\log C_0
 \leq\int_{\frac{1}{2}}^{\frac{21}{20}}\frac{\beta(t)}t\,d t
 \leq\beta\bigl(\frac{21}{20}\bigr)\log\bigl(\frac{21}{10}\bigr).
\end{equation*}
This proves \eqref{eq:frequency-lower}. Increasing the threshold for $\NN$
and decreasing the positive coefficient gives
\begin{equation}\label{eq:frequency-large}
 \beta(r)\geq c_n\NN\qquad \text{ for }\frac{21}{20}\leq r<4.
\end{equation}
Insert the intermediate radius $\frac{299}{200}$ between $\frac{149}{100}$ and $\frac{3}{2}$.
Equations \eqref{eq:frequency-ratio}, \eqref{eq:poisson-comparison}, and
\eqref{eq:frequency-large} give
\begin{align}
 M\bigl(\frac{149}{100}\bigr)
 &\leq C_1L\bigl(\frac{299}{200}\bigr)\notag\\
 &\leq C_1L\bigl(\frac{3}{2}\bigr)
        \exp\left(-c_n\NN\log\frac{300}{299}\right)\notag\\
 &\leq C_1M\bigl(\frac{3}{2}\bigr)
        \exp\left(-c_n\NN\log\frac{300}{299}\right).
 \label{eq:inner-decay-before-absorption}
\end{align}
Take $b=\frac{c_n}{2}\log\bigl(\frac{300}{299}\bigr)>0$. If $\NN\geq \frac{\log C_1}{b}$, then
$C_1e^{-2b\NN}\leq e^{-b\NN}$. Thus
\eqref{eq:inner-decay-before-absorption} proves \eqref{eq:inner-decay} after
increasing $N_1$. The first half of the radial gap permits the comparison of $M$ with $L$, and the second half supplies the frequency integral.
\end{proof}

\begin{lemma}\label{lem:outer-growth}
Assume \eqref{eq:stable} and $\NN\geq1$. Then
\begin{equation}\label{eq:outer-growth}
 M(4)\leq e^{C_{n,S}\NN}M\bigl(\frac{3}{2}\bigr),
\end{equation}
where $C_{n, S}$ is a positive constant depending only on $n$ and $S$.
\end{lemma}

\begin{proof}
By \eqref{eq:stable},
\begin{equation}\label{eq:stable-ratio}
    \frac{M(4)}{M(2)}\le 2^{S\NN}.
\end{equation}
The missing comparison in \eqref{eq:stable-ratio} is from radius $\frac{3}{2}$ to
radius $2$. It follows by interpolation rather than by replacing $M(2)$
by the smaller quantity $M(\frac{3}{2})$. Define
\begin{equation*}\label{eq:theta}
 \theta=\frac{\log(\frac{14}{9})}{\log(\frac{7}{3})}\in(0,1), 
 \text{ equivalently, }\log\frac94=\theta\log\frac32+(1-\theta)\log\frac72.
\end{equation*}
Using \eqref{eq:three-sphere}, \eqref{eq:poisson-comparison}, and
$L(r)\leq M(r)$, we obtain
\begin{equation}\label{eq:outer-interpolation}
 M(2)\leq C_nL\bigl(\frac{9}{4}\bigr)
 \leq C_nL\bigl(\frac{3}{2}\bigr)^\theta L\bigl(\frac{7}{2}\bigr)^{1-\theta}
 \leq C_nM\bigl(\frac{3}{2}\bigr)^\theta M(4)^{1-\theta}.
\end{equation}
Combining \eqref{eq:stable-ratio} with \eqref{eq:outer-interpolation}, we have
\begin{equation}\label{eq:outer-absorption}
 M(4)\leq C_n2^{S\NN}M\bigl(\frac{3}{2}\bigr)^\theta M(4)^{1-\theta},\qquad
 \left(\frac{M(4)}{M\bigl(\frac{3}{2}\bigr)}\right)^\theta\leq C_n2^{S\NN}.
\end{equation}
All suprema in \eqref{eq:outer-absorption} are finite and positive. It further gives
\begin{equation*}\label{eq:outer-explicit}
 \frac{M(4)}{M\bigl(\frac{3}{2}\bigr)}
 \leq\exp\left(\frac{\log C_n}{\theta}
              +\frac{S\log2}{\theta}\NN\right)
 \leq e^{C_{n,S}\NN},
\end{equation*}
where the last inequality uses $\NN\geq1$. This proves
\eqref{eq:outer-growth}.
\end{proof}

\begin{lemma}\label{lem:directions}
Assume \eqref{eq:stable} and normalize $M\bigl(\frac{3}{2}\bigr)=1$. For every fixed
$\eps>0$, there are $\delta>0$ and $N_2$, depending only on
$n,S,\eps$, such that for $\NN\geq N_2$,
\begin{equation}\label{eq:good-directions}
\sigma(E)\geq\delta,
\end{equation}
where 
$$ E=\{\xi\in\Sph^{n-1}:\bigl|u\bigl(\frac{3}{2}\xi\bigr)\bigr|\geq e^{-\eps \NN}\}.$$
For every $\xi\in\Sph^{n-1}$, the restriction $r\mapsto u(r\xi)$ extends to a
holomorphic function $f_\xi$ in $\D_3$, which is real-valued on $(-3,3)$, with
\begin{equation}\label{eq:radial-hol-growth}
 \sup_{|z|\leq\frac52}|f_\xi(z)|\leq e^{C_{n,S}\NN}
\end{equation}
for $\NN$ above a dimensional and $S$-dependent threshold. The constant in
\eqref{eq:radial-hol-growth} is independent of $\xi$.
\end{lemma}

\begin{proof}
The argument combines the classical spherical harmonic expansion and
reproducing kernel \cite[Chapter 5]{ABR} with the trigonometric Remez type
inequality \cite[Theorem 2]{Erdelyi1992Remez}.

{\it Step 1. Coefficient bounds and radial holomorphic extension.}
Let $d_k$ be the dimension of the space of spherical harmonic functions of degree $k$ on $\Sph^{n-1}$. One has
$$d_k\leq C_n(1+k)^{n-2}.$$ For a real orthonormal basis
$Y_1,\ldots,Y_{d_k}$ in $L^2(d\mu)$, the diagonal sum
$\sum_jY_j(\xi)^2$ is independent of the basis and invariant under rotations. It is constant on the sphere, and its integral is
\begin{equation*}\label{eq:kernel-diagonal}
 \int_{\Sph^{n-1}}\sum_{j=1}^{d_k}Y_j(\xi)^2\dd\mu(\xi)
 =\sum_{j=1}^{d_k}\|Y_j\|_2^2=d_k.
\end{equation*}
So, $\sum_{j=1}^{d_k}Y_j(\xi)^2=d_k.$ The Cauchy--Schwarz inequality now gives, for every spherical harmonic function $Q=\sum_{j=1}^{d_k}c_jY_j$, 
\begin{equation}\label{eq:point-evaluation}
 |Q(\xi)|^2=\left|\sum_jc_jY_j(\xi)\right|^2
 \leq\left(\sum_jc_j^2\right)\left(\sum_jY_j(\xi)^2\right)=d_k\|Q\|_2^2.
\end{equation}
Recall $u=\sum_{k\geq0}P_k$ is its expansion in homogeneous harmonic polynomials. On the unit sphere, the function $3^kP_k(\cdot)$ is the orthogonal projection of $u(3\cdot)$
onto the space of spherical harmonic functions of degree $k$ in $L^2(d\mu)$. Thus,
$$ \|3^kP_k\|_{L^2(d\mu)}\leq L(3)\leq M(3),$$ and by \eqref{eq:point-evaluation},
\begin{equation}\label{eq:Pk-bound}
 |P_k(\xi)|\leq C_n(1+k)^{\frac{n-2}{2}}3^{-k}M(3)
 \qquad\text{ for any } \xi\in\Sph^{n-1}.
\end{equation}
Define
\begin{equation*}\label{eq:radial-series}
 f_\xi(z)=\sum_{k\geq0}P_k(\xi)z^k.
\end{equation*}
By \eqref{eq:Pk-bound}, it converges
locally uniformly in $\D_3$. The coefficients are real, and homogeneity gives
$f_\xi(r)=\sum_kP_k(r\xi)=u(r\xi)$ for $-3<r<3$. Moreover,
\begin{equation*}\label{eq:radial-series-bound}
 \sup_{|z|\leq\frac{5}{2}}|f_\xi(z)|
 \leq C_nM(3)\sum_{k\geq0}(1+k)^{\frac{n-2}{2}}\bigl(\frac{5}{6}\bigr)^k
 \leq C_nM(3)\leq e^{C_{n,S}\NN},
\end{equation*}
where the last step uses \eqref{eq:outer-growth} and $M\bigl(\frac{3}{2}\bigr)=1$.
This proves \eqref{eq:radial-hol-growth}. Only the radial variable is
complexified; no several-variable holomorphic extension is needed.

{\it Step 2. Polynomial approximation with degree proportional to $\NN$.}
Choose $m=\lceil\Lambda \NN\rceil$, where $\Lambda\geq1$ will be fixed in terms of $n,S,\eps$, and set
\begin{equation*}\label{eq:spherical-truncation}
 p_m(\xi)=\sum_{k=0}^m\bigl(\frac{3}{2}\bigr)^kP_k(\xi).
\end{equation*}
Equations \eqref{eq:Pk-bound} and \eqref{eq:outer-growth} give
\begin{align*}
 \|u\bigl(\frac{3}{2}\,\cdot\bigr)-p_m\|_\infty
 &\leq C_nM(3)\sum_{k>m}(1+k)^{\frac{n-2}{2}}2^{-k}\notag\\
 &\leq C_ne^{C_{n,S}\NN}(m+1)^{\frac{n-2}{2}}2^{-m}.
 \label{eq:truncation-tail}
\end{align*}
Choose $\Lambda=\Lambda(n, S, \eps)$ sufficiently large.
After increasing the threshold for $\NN$, the polynomial factor can be absorbed in the exponential, and we obtain
\begin{equation}\label{eq:approximation}
 \|u\bigl(\frac{3}{2}\,\cdot\bigr)-p_m\|_\infty\leq e^{-2\eps \NN},
 \qquad \|p_m\|_\infty\geq1-e^{-2\eps \NN}\geq\frac12.
\end{equation}
The equality $\|u\bigl(\frac{3}{2}\,\cdot\bigr)\|_\infty=1$ follows from the maximum principle
and the normalization. 

{\it Step 3. A Remez type estimate on great circles through a maximum.}
We use the following Remez type estimate for trigonometric polynomials; see \cite[Theorem 2]{Erdelyi1992Remez}.
There is an absolute constant $C_{\mathrm 1}\geq1$ such that, whenever a
real-valued trigonometric polynomial $\mathcal T$ of degree at most $m$ satisfies
$|\mathcal T|\leq t$ on a measurable set $I\subset[0,2\pi)$ with
$|I|\geq2\pi-s$, where $t>0$ and $0<s\leq\frac12$, one has
\begin{equation}\label{eq:circle-propagation}
 \|\mathcal T\|_\infty\leq te^{C_1ms}.
\end{equation}
This follows by applying the cited theorem to $t^{-1}\mathcal T$, and its range $0<s\leq\frac\pi2$ includes the one used here. We only need
\eqref{eq:circle-propagation} with an absolute constant; for the sharp trigonometric form, see \cite[Theorem 1.1]{TY}. The dependence on the exceptional length $s$ in the exponent
is important, since a fixed loss $e^{Cm}$ would not suffice for arbitrarily
small $\eps$.

Choose $\xi_0\in\Sph^{n-1}$ with $|p_m(\xi_0)|=\|p_m\|_\infty$.
For a unit vector $\nu\perp \xi_0$, put
\begin{equation*}\label{eq:great-circle}
 \Gamma_\nu(t)=(\cos t)\xi_0+(\sin t)\nu,
 \qquad Q_\nu(t)=p_m(\Gamma_\nu(t)),\qquad 0\leq t<2\pi.
\end{equation*}
Substitution of linear combinations of $\cos t$ and $\sin t$ into a
polynomial of degree $m$ shows that $Q_\nu$ is a real-valued trigonometric polynomial of
degree at most $m$. Since $\Gamma_\nu(0)=\xi_0$,
\begin{equation}\label{eq:circle-maximum}
 \|Q_\nu\|_\infty=\|p_m\|_\infty\geq\frac12
 \quad\text{for every }\nu.
\end{equation}
Choose $s\in(0,\frac12)$ so small that $C_1(\Lambda+1)s<\frac{\eps}{2}$ and define
\begin{equation}\label{eq:circle-large-set}
 G_\nu=\{t\in[0,2\pi):|Q_\nu(t)|\geq2e^{-\eps \NN}\}.
\end{equation}
If $|G_\nu|<s$, apply \eqref{eq:circle-propagation} to its complement. Since
$m\leq(\Lambda+1)\NN$, equations \eqref{eq:circle-maximum} and
\eqref{eq:circle-large-set} imply
\begin{equation*}\label{eq:remez-contradiction}
 \frac12\leq\|Q_\nu\|_\infty
 \leq2e^{-\eps \NN+C_1ms}
 \leq2e^{-\frac{\eps \NN}{2}},
\end{equation*}
which is impossible for large $\NN$. Hence
\begin{equation}\label{eq:circle-length}
 |G_\nu|\geq s\qquad\text{for every }\nu\in\Sph^{n-2}\subset \xi_0^\perp.
\end{equation}

{\it Step 4. From circle length to spherical area.}
For any measurable $G\subset[0,2\pi)$ with $|G|\geq s$, remove
\begin{equation}\label{eq:polar-degeneracy}
 U=[0,\frac s8)\cup(\pi-\frac s8,\pi+\frac s8)
       \cup(2\pi-\frac s8,2\pi),\qquad |U|=\frac{s}{2}.
\end{equation}
Then $|G\setminus U|\geq \frac s2$, and
$|\sin t|\geq \frac{s}{4\pi}$ on $G\setminus U$. Consequently,
\begin{equation}\label{eq:weighted-circle}
 \int_G|\sin t|^{n-2}\dd t
 \geq\frac s2\left(\frac{s}{4\pi}\right)^{n-2}
 \geq c_ns^{n-1}.
\end{equation}
Define
\begin{equation}\label{eq:poly-large-set}
 E_m=\{\xi\in\Sph^{n-1}:|p_m(\xi)|\geq2e^{-\eps \NN}\}.
\end{equation}
The domain $0<t<\pi$ covers the sphere except its poles
once, whereas $0<t<2\pi$ covers it twice, because
\begin{equation}\label{eq:double-cover}
 \Gamma_\nu(t)=\Gamma_{-\nu}(2\pi-t).
\end{equation}
Equations \eqref{eq:circle-length}--\eqref{eq:double-cover} therefore give
\begin{equation}\label{eq:spherical-area}
 \begin{split}
 2\sigma(E_m)
 &=\int_{\Sph^{n-2}}\int_0^{2\pi}
       \ind_{E_m}(\Gamma_\nu(t))|\sin t|^{n-2}\dd t\dd\sigma(\nu)\\
 &=\int_{\Sph^{n-2}}\int_{G_\nu}|\sin t|^{n-2}\dd t\dd\sigma(\nu)
 \geq c_ns^{n-1}.
 \end{split}
\end{equation}
Here both surface measures are the usual, unnormalized measures, and the
fixed area of $\Sph^{n-2}$ is included in $c_n$. No antipodal symmetry of
$E_m$ is needed for \eqref{eq:double-cover}.
Finally, by \eqref{eq:approximation}, we obtain
\begin{equation*}\label{eq:large-set-transfer}
 \bigl|u\bigl(\frac{3}{2} \xi\bigr)\bigr|\geq|p_m(\xi)|-e^{-2\eps \NN}
 \geq2e^{-\eps \NN}-e^{-2\eps \NN}
 \geq e^{-\eps \NN}\qquad\text{ for any } \xi\in E_m.
\end{equation*}
Thus $E_m\subset E$ and \eqref{eq:spherical-area} proves
\eqref{eq:good-directions}. The choices are made in the order
$\Lambda$, $s$ and $\delta=c_ns^{n-1}$, and then the threshold for $\NN$.
In particular, $\delta$ is independent of $u$, $\xi$, and $\NN$. This completes the proof.
\end{proof}
 Now we are ready to prove Proposition \ref{prop:stable}.
\begin{proof}[Proof of Proposition \ref{prop:stable}]
It suffices first to treat large $\NN$. Normalize $M\bigl(\frac{3}{2}\bigr)=1$, and use
Lemma \ref{lem:directions} with $\eps=\frac b8$, where $b$ is from Lemma \ref{lem:inner-decay}. For each $\xi\in E$, define
\begin{equation*}\label{eq:shifted-radial}
 \phi_\xi(\zeta)=f_\xi\left(\frac32+\frac\zeta4\right).
\end{equation*}
Note
$\bigl|\frac{3}{2}+\frac{\zeta}{4}\bigr|<\frac{5}{2}$ on $\D_2$. By Lemma \ref{lem:directions}, $\phi_\xi$ is holomorphic in $\D_2$ and real-valued on $(-2,2)$. Moreover, \eqref{eq:radial-hol-growth} gives
\begin{equation}\label{eq:shifted-growth}
 \sup_{\D_2}|\phi_\xi|\leq e^{C_{n,S}\NN},\quad
 |\phi_\xi(0)|\geq e^{-\frac{b\NN}{8}}.
\end{equation}
Use the fixed intervals $(-\frac{2}{25}, \frac{2}{25})$ and $[-\frac{7}{100},-\frac{3}{50}]$.
The latter is compactly contained in the former with the origin removed.
For $\zeta$ in the latter interval,
$$ \frac32+\frac\zeta4
 \in\left[\frac{593}{400},\frac{297}{200}\right]
 \subset\left(0,\frac{149}{100}\right).$$ Thus, \eqref{eq:inner-decay} gives
\begin{equation}\label{eq:shifted-inner-interval}
 \sup_{\bigl[-\frac{7}{100},-\frac{3}{50}\bigr]}|\phi_\xi|\leq e^{-b\NN}.
\end{equation} Since $\frac{2}{25}<\frac{1}{4}$ and $\frac{b}{8}<b$,
Corollary \ref{cor:separation}, with $T=\NN$ and $q=\frac 12$, applies to
\eqref{eq:shifted-growth} and \eqref{eq:shifted-inner-interval}. All its
parameters are fixed independently of $\xi,\NN$. It yields
\begin{equation}\label{eq:radial-corollary}
 \int_{-\frac{2}{25}}^{\frac{2}{25}}
 \left|\frac{\phi'_\xi(\zeta)}{\NN\phi_\xi(\zeta)}\right|^{1/2}\dd\zeta
 \geq\kappa_{n,S}>0.
\end{equation}
Since
\begin{equation*}\label{eq:radial-change}
 \frac{\phi'_\xi(\zeta)}{\phi_\xi(\zeta)}
 =\frac14\frac{\partial}{\partial r}\log|u(r\xi)|,
\end{equation*}
\eqref{eq:radial-corollary} implies
\begin{equation}\label{eq:radial-lower}
\int_{\frac{37}{25}}^{\frac{38}{25}}\bigl|\partial_r\log|u(r\xi)|\bigr|^{\frac 12}\,d r
 \geq\frac{\kappa_{n,S}}2\NN^{\frac 12}\qquad\text{ for any }\xi\in E.
\end{equation}
The radial function is not identically zero because $\phi_\xi(0)\ne0$ and \eqref{eq:radial-lower} is a well-defined integral despite the zeros since $0<q<1$.

Let $g=|\nabla\log|u||$ outside $\{u=0\}$, extended by zero on $\{u=0\}$.
Since $[\frac{37}{25},\frac{38}{25}]\subset(0,\frac 85)$, \eqref{eq:radial-lower} implies
\begin{align*}
 \int_{B_{\frac 8 5}}g^{\frac1 2}\,d x
 &\geq\int_E\int_{\frac{37}{25}}^{\frac{38}{25}}g(r\xi)^{\frac 1 2}r^{n-1}\,d r\dd\sigma(\xi)
 \notag\\
 &\geq\left(\frac{37}{25}\right)^{n-1}
       \frac{\kappa_{n,S}}2\NN^{\frac12}\sigma(E)
 \geq c_{n,S}\NN^{\frac1 2}.\label{eq:gradient-lower}
\end{align*}
The positive lower bound on $\sigma(E)$ is indispensable, since one radial restriction would have zero angular measure.

Proposition~\ref{prop:gradient-upper} at exponent $\frac12$ now gives
\begin{equation}\label{eq:stable-comparison}
 c_{n,S}\NN^{\frac{1}{2}}
 \leq\int_{B_{\frac 85}}g^{\frac12}\dd x
 \leq C_n\bigl(1+\HH^{n-1}(\{u=0\}\cap B_2)\bigr)^{\frac12}.
\end{equation}
Squaring \eqref{eq:stable-comparison} proves
\eqref{eq:stable-conclusion} for large $\NN$. Below the fixed threshold,
note $\NN\leq N_*(n,S)(1+\HH^{n-1}(\{u=0\}\cap B_2))$. Then we can increase $C_{n,S}$ to get the desired estimate.
\end{proof}

\section{Proof of Theorem \ref{thm:main}}\label{sec:main-proof}
We use the stable ball construction of Logunov, Priya and Sartori \cite[Lemma 6.5]{LPS}. The following version records the inner index lower bound needed to apply Proposition~\ref{prop:stable}. It is a direct consequence of their stated conclusions, not an additional hypothesis.

\begin{lemma}[Stable ball lemma, {\cite[Lemma 6.5]{LPS}}]\label{lem:stable-ball}
There are dimensional constants $L_n>0$, $S_n\geq1$, and $c_n>0$ with the
following property. Let $u$ be a non-zero harmonic function in $B_4$. Suppose $0<w<\frac{1}{20}$, $\rho>0$, and
\begin{equation*}\label{eq:shell-assumption}
 \mathcal S_{\rho,w}=\{x:\rho-10w\leq|x|\leq\rho+10w\}
 \subset B_2\setminus B_1.
\end{equation*}
Assume, for some $v\geq1$, that
\begin{equation*}\label{eq:shell-frequency}
 v\leq\beta(t)\leq10v
 \qquad\text{ for } t\in [\rho-10w, \rho+10w].
\end{equation*}
If
\begin{equation}\label{eq:width-condition}
 \frac{v w}{\log(\frac 1 w)}\geq L_n,
\end{equation}
then there is a ball $D$ of radius $w$ such that
$\overline{4D}\subset\mathcal S_{\rho,w}$ and
\begin{equation}\label{eq:stable-ball-indices}
 \NN\bigl(\frac D2\bigr)\geq c_nv w,\qquad
 \NN(2D)\leq S_n\NN\bigl(\frac D 2\bigr).
\end{equation}
No condition is imposed on $u$ at the center of $D$.
\end{lemma}

\begin{proof}
The frequency parameter denoted by $N$ in \cite[Lemma 6.5]{LPS} is
$v$ here. Its dimensional largeness condition is
\eqref{eq:width-condition} after $L_n$ is fixed sufficiently large. Its
conclusions include $4D\subset\mathcal S_{\rho,w}$ and
\begin{equation}\label{eq:stable-ball-cited}
 cv w\leq\NN(2D)\leq C\NN\bigl(\frac D2\bigr),
\end{equation}
where $c$ and $C$ are positive constants depending only on $n$. Inequality \eqref{eq:stable-ball-cited}
implies \eqref{eq:stable-ball-indices} with $c_n=\frac cC$ and
$S_n=\max\{1,C\}$. The shell is closed, so the inclusion also holds for $\overline{4D}$. This imposes no vanishing condition at the selected center.
\end{proof}

Theorem \ref{thm:main} is a direct consequence of the following theorem.
\begin{theorem}\label{thm:main-2}
Let $n\geq3$. There is a constant $C_n>0$, depending only on $n$, such that
every non-zero harmonic function $u$ in the open ball $B_4$ satisfies
\begin{equation}\label{eq:main-additive}
 \NN\leq C_n\bigl(1+\HH^{n-1}\bigl(\{u=0\}\cap B_2\bigr)\bigr).
\end{equation}
If, in addition, $u(0)=0$, then
\begin{equation}\label{eq:main-linear}
 \HH^{n-1}\bigl(\{u=0\}\cap B_2\bigr)\geq c_n\NN
\end{equation}
for a constant $c_n>0$ depending only on $n$.
\end{theorem}

\begin{proof}
It suffices first to prove $$\HH^{n-1}\bigl(\{u=0\}\cap B_2\bigr)\geq c_n\NN$$ for $\NN$ above a fixed dimensional
threshold, without assuming $u(0)=0$. We may suppose $\HH^{n-1}\bigl(\{u=0\}\cap B_2\bigr)<\infty$.

{\it Step 1. Frequency layers and separated middle shells.}
By \eqref{eq:frequency-lower}, the monotonicity of $\beta$ established in \eqref{eq:frequency-monotonicity}, after increasing the dimensional threshold for $\NN$, we have
\begin{equation}\label{eq:nu0}
 v_0:=\beta\bigl(\frac{11}{10}\bigr)\geq\beta\bigl(\frac{21}{20}\bigr)\geq c_n\NN,
 \qquad v_0\geq1.
\end{equation}
For $j=0,1,\ldots$, define
\begin{equation}\label{eq:frequency-layers}
 v_j=10^jv_0,\qquad
 I_j=\left\{t\in\left[\frac{11}{10},\frac{19}{10}\right]:
             v_j\leq\beta(t)<10v_j\right\}.
\end{equation}
Since $\beta$ is continuous and nondecreasing, each $I_j$ is an interval, possibly empty. Only finitely many are nonempty because $\beta\bigl(\frac{19}{10}\bigr)$ is finite. This number need not be uniformly bounded. Endpoint conventions
have no effect on
\begin{equation*}\label{eq:layer-lengths}
 \sum_j|I_j|=\frac{19}{10}-\frac{11}{10}=\frac45.
\end{equation*}
Discard zero-length intervals. For each remaining $I_j$, let $\rho_j$ be
its midpoint and define  $$w_j=\frac{|I_j|}{40},\qquad \mathcal S_j=\{x:\rho_j-10w_j\leq|x|\leq\rho_j+10w_j\}.$$ Then,
\begin{equation}\label{eq:middle-shells}
 [\rho_j-10w_j,\rho_j+10w_j]\Subset{\rm int}(I_j),
 \qquad
 \mathcal S_j\subset B_2\setminus B_1,
\end{equation}
and the shells $\mathcal S_j$ are pairwise separated. Moreover,
\begin{equation}\label{eq:width-sum}
 0<w_j\leq
 \sum_jw_j=\frac1{50},\qquad
 v_j\leq\beta(t)\leq10v_j\quad
 \text{ for } t\in[\rho_j-10w_j, \rho_j+10w_j].
\end{equation}
Thus the frequency on each middle shell is controlled by a fixed factor,
although the frequency on the whole annulus need not be.

{\it Step 2. Retain enough width while enforcing the scale condition.}
Fix the constants in Lemma~\ref{lem:stable-ball}, and let
$C_*=C_{n,S_n}$ be the constant in Proposition \ref{prop:stable}. Choose a dimensional constant $Q\geq1$ such that
\begin{equation*}\label{eq:Q-choice}
 \mathrm eQ\geq L_n,\qquad50c_nQ\geq2C_*.
\end{equation*}
Retain the set of indices
\begin{equation}\label{eq:retained-indices}
 \mathcal J=\{j:v_jw_j^2\geq Q\}.
\end{equation}
This is a selection within the proof, not an additional assumption on $u$.
For a discarded index $j$, \eqref{eq:frequency-layers} and \eqref{eq:retained-indices} give
\begin{equation*}\label{eq:discard-one}
 w_j<\sqrt{\frac Q{v_j}}
 =\sqrt{\frac Q{v_0}}\,10^{-\frac j2}.
\end{equation*}
Consequently,
\begin{equation}\label{eq:discard-width}
 \sum_{j\notin\mathcal J}w_j
 \leq\sqrt{\frac Q{v_0}}\sum_{j=0}^\infty10^{-\frac j2}
 =\frac{1}{1-10^{-\frac12}}\sqrt{\frac{Q}{v_0}}
 \leq\frac1{100},
\end{equation}
provided $v_0\geq 10^4\frac{Q}{\bigl(1-10^{-\frac12}\bigr)^2}$. By \eqref{eq:nu0}, this last
condition is achieved by increasing only the dimensional threshold for
$\NN$. Equations \eqref{eq:width-sum} and \eqref{eq:discard-width} imply
\begin{equation}\label{eq:retained-width}
 \sum_{j\in\mathcal J}w_j
 \geq\frac1{50}-\frac1{100}=\frac1{100}.
\end{equation}
For every retained index $j$,
\begin{equation*}\label{eq:retained-size}
 v_jw_j\geq\frac Q{w_j}\geq50Q,
 \qquad
 \frac{v_jw_j}{\log\bigl(\frac{1}{w_j}\bigr)}
 \geq\frac Q{w_j\log\bigl(\frac{1}{w_j}\bigr)}\geq\mathrm eQ\geq L_n.
\end{equation*}
The penultimate inequality follows from
$-t\log t\leq {\rm e}^{-1}$ for $0<t<1$. In addition,
$20w_j\leq \frac25<\rho_j$ and $w_j\leq\frac{1}{50}<\frac{1}{20}$.
Thus every retained shell meets all the assumptions of
Lemma~\ref{lem:stable-ball}.

{\it Step 3. Stable balls and their nodal volume contributions.}
For each $j\in\mathcal J$, choose $D_j=B(x_j,w_j)$ from
Lemma~\ref{lem:stable-ball} and write $a_j=\NN\bigl(\frac{D_j}{2}\bigr)$. Then
\begin{equation}\label{eq:chosen-balls}
 \overline{4D_j}\subset\mathcal S_j\Subset B_2,\qquad
 a_j\geq c_nv_jw_j\geq2C_*,\qquad
 \NN(2D_j)\leq S_na_j.
\end{equation}
The closed balls $\overline{4D_j}$ are pairwise disjoint by
\eqref{eq:middle-shells}.
Define $U_j(y)=u(x_j+w_jy)$. Since $\overline{4D_j}\Subset B_2\Subset B_4$,
$U_j$ is harmonic in an open neighborhood of $\overline{B_4}$, even though
$u$ was only assumed harmonic in the original open ball. The doubling indices of $U_j$ obey
\begin{equation}\label{eq:rescaled-indices}
 \NN_{U_j}(B_{\frac12})=a_j,
 \qquad \NN_{U_j}(B_2)=\NN(2D_j)\leq S_na_j.
\end{equation}
The zero sets correspond exactly under the same dilation:
\begin{equation}\label{eq:rescaled-zero-set}
\begin{aligned}
 \{U_j=0\}\cap B_2
 =\left\{\frac{x-x_j}{w_j}:x\in \{u=0\}\cap2D_j\right\},\\
 \HH^{n-1}\bigl(\{U_j=0\}\cap B_2\bigr)
 =w_j^{1-n}\HH^{n-1}\bigl(\{u=0\}\cap2D_j\bigr).
 \end{aligned}
\end{equation}

Apply Proposition~\ref{prop:stable} to \eqref{eq:rescaled-indices}, and use
\eqref{eq:rescaled-zero-set}:
\begin{equation*}\label{eq:rescaled-stable}
 a_j\leq C_*\left(1+
        \frac{\HH^{n-1}\bigl(\{u=0\}\cap2D_j\bigr)}{w_j^{n-1}}\right).
\end{equation*}
Since $a_j\geq2C_*$ by \eqref{eq:chosen-balls}, the additive constant is
absorbed as follows:
\begin{align}
 \HH^{n-1}(\{u=0\}\cap2D_j)
 &\geq w_j^{n-1}\left(\frac{a_j}{C_*}-1\right)
 \geq\frac{a_j}{2C_*}w_j^{n-1}\notag\\
 &\geq\frac{c_n}{2C_*}v_jw_j^n.\label{eq:local-layer-volume}
\end{align}
The extra factor $w_j$ in $w_j^n$ comes from the local index lower bound
$a_j\geq c_nv_jw_j$, not from the scaling of area.

{\it Step 4. Weighted summation without a layer-count loss.}
The balls $2D_j$ are disjoint subsets of $B_2$, so
\eqref{eq:local-layer-volume} and \eqref{eq:frequency-layers} give
\begin{equation}\label{eq:weighted-volume}
 \HH^{n-1}\bigl(\{u=0\}\cap B_2\bigr)\geq\sum_{j\in\mathcal J}\HH^{n-1}\bigl(\{u=0\}\cap2D_j\bigr)
 \geq c_nv_0\sum_{j\in\mathcal J}10^jw_j^n.
\end{equation}
To use the width lower bound \eqref{eq:retained-width}, retain the frequency
weights. H\"older's inequality gives
\begin{equation}\label{eq:weighted-holder}
 \begin{split}
 \sum_{j\in\mathcal J}w_j
 &=\sum_{j\in\mathcal J}(10^jw_j^n)^{\frac 1n}10^{-\frac jn}\\
 &\leq\left(\sum_{j\in\mathcal J}10^jw_j^n\right)^{\frac1n}
       \left(\sum_{j=0}^\infty10^{-\frac{j}{n-1}}\right)^{\frac{n-1}{n}}.
 \end{split}
\end{equation}
Then,
\eqref{eq:retained-width} and \eqref{eq:weighted-holder} imply
\begin{equation}\label{eq:weighted-sum-lower}
 \sum_{j\in\mathcal J}10^jw_j^n
 \geq\left(\sum_{j\in\mathcal J}w_j\right)^n
          (1-10^{-\frac{1}{n-1}})^{n-1}
 \geq100^{-n}(1-10^{-\frac{1}{n-1}})^{n-1}>0.
\end{equation}
This bound is independent of the number and highest frequency of the
layers. Equations
\eqref{eq:nu0}, \eqref{eq:weighted-volume}, and
\eqref{eq:weighted-sum-lower} now give
\begin{equation}\label{eq:large-N-conclusion}
 \HH^{n-1}\bigl(\{u=0\}\cap B_2\bigr)\geq c_nv_0\geq c_n\NN
\end{equation}
for all $\NN$ above a fixed dimensional threshold.

{\it Step 5. Small indices and a zero at the origin.}
Let $N_*(n)$ dominate the dimensional thresholds already chosen. For $\NN<N_*(n)$,
$$\NN\leq N_*(n)\bigl(1+ \HH^{n-1}\bigl(\{u=0\}\cap B_2\bigr)\bigr).$$ Together with
\eqref{eq:large-N-conclusion}, this proves \eqref{eq:main-additive} for all
$\NN$ after increasing $C_n$.

If $u(0)=0$, the local nodal lower bound \eqref{eq:nadirashvili} in $B_1$
gives
$$\HH^{n-1}\bigl(\{u=0\}\cap B_2\bigr)\geq\HH^{n-1}(\{u=0\}\cap B_1)\geq c_1(n)>0.$$
Hence,
\begin{equation}\label{eq:remove-additive}
1+ \HH^{n-1}\bigl(\{u=0\}\cap B_2\bigr)\leq(1+c_1(n)^{-1}) \HH^{n-1}\bigl(\{u=0\}\cap B_2\bigr).
\end{equation}
Combining \eqref{eq:remove-additive} with \eqref{eq:main-additive} proves
\eqref{eq:main-linear}. The condition $u(0)=0$ is used only at this last
step.
\end{proof}

\section{Proof of Theorem \ref{thm:main-posi} and optimality of its logarithmic order}\label{sec:main-posi-proof}

Following an observation of Greilhuber, Logunov and Sodin, we show in this section that the method developed above for the estimate of nodal volume can also be adapted to control the volume of the positivity set. The main additional ingredient is the analytic small cube estimate of Donnelly and Fefferman \cite[Proposition 5.11]{DonnellyFefferman1988}, in the form recorded in \cite[Proposition C.1]{LPS}. Combining this estimate with the fractional moment upper bound in Proposition \ref{prop:gradient-upper} and the left hand side inequality in \eqref{eq:stable-comparison}, we first obtain the desired estimate under the stable growth assumption. We then use the same frequency layer decomposition as in the proof of Theorem \ref{thm:main-2}, with a modified final summation, to prove Theorem \ref{thm:main-posi} in general. At the end of this section, we show that the logarithmic order is optimal up to dimensional constants.

Theorem \ref{thm:main-posi} is a direct consequence of the following theorem.
\begin{theorem}\label{sv:thm:main}
Let $n\geq3$. There is a constant $c_n>0$, depending only on $n$, such that
every non-zero harmonic function $u$ in the open ball $B_1$ with $u(0)=0$ satisfies
\begin{equation}\label{sv:eq:main}
\HH^{n}\bigl(\{ u>0\}\cap B_{\frac{1}{2}}\bigr)
\ge c_n\bigl(\log(1+\NN)\bigr)^{1-n}.
\end{equation}
\end{theorem}
We begin with some useful results.
\begin{lemma}\label{sv:lem:balance}
Let $B=B(z,R)$, and let $f$ be a non-zero harmonic function in $2B$ with
$f(z)=0$. If there is a positive constant $d_0$ such that
$$
\int_{2B}f^2\,dx\le d_0\int_Bf^2\,dx<\infty,
$$
then
\begin{equation}\label{sv:eq:balance}
\HH^n(\{f>0\}\cap B)
\ge \frac{\VV(B)}{2d_0}.
\end{equation}
\end{lemma}

\begin{proof}
Set $M=\sup_B|f|$. Subharmonicity of $f^2$ and
$B(x,R)\subset2B$ for any $x\in B$ give
$$f(x)^2\le\frac{1}{\VV(B)}\int_{B(x, R)}f^2\,dx\le\frac{1}{\VV(B)}\int_{2B}f^2\,dx\le\frac{d_0}{\VV(B)}\int_{B}f^2\,dx,$$
which implies
\begin{equation}\label{eq:sv-1}
M^2\le \frac{d_0}{\VV(B)}\int_{B}f^2\,dx.
\end{equation}
Since $f(z)=0$, the mean value property implies $\int_Bf\,dx=0$. Thus, writing
$f_+=\max\{f,0\}$, we know
\begin{equation}\label{eq:sv-2}
M\HH^n(\{f>0\}\cap B)
\ge\int_Bf_+\,dx
=\frac12\int_B|f|\,dx\ge\frac{\int_Bf^2\,dx}{2M}.
\end{equation}
Combining \eqref{eq:sv-1} and \eqref{eq:sv-2} proves \eqref{sv:eq:balance}.
\end{proof}

The additional ingredient we require is the following
analytic small cube estimate of Donnelly and Fefferman \cite[Proposition 5.11]{DonnellyFefferman1988}, as formulated
in \cite[Proposition C.1]{LPS}.
\begin{proposition}[{\cite[Proposition C.1]{LPS}}]\label{prop:C1}
Let $k$ be a sufficiently large integer depending on $n$. Let $G(z)$ be a non-zero function holomorphic in $\{z\in \C^n: |z|\le 3^k\}$, and satisfy
$$
\max_{|z|\le 2^k}|G(z)|
\le |G(0)|e^d,
$$
for some $d\ge 1$ sufficiently large depending on $n$. Assume that
$G(z)$ is real and nonnegative for real
$
x\in \QQ_0
:=
\bigl\{(x_1,\ldots,x_n)\in\R^n:
|x_i|\le 1,\ 1\le i\le n\bigr\}.
$
Suppose that $\RR\subset \frac12 \QQ_0$ is a sub-cube and subdivide $\RR$
into equal sub-cubes $\RR_i$ of side length
$\asymp_n \frac{1}{d}.$ Let $\eps>0$ be given. Then there exists a subset
$E\subset \QQ_0$ of measure $\HH^n(E)\le \eps$ and a constant $C_{n, \eps}\ge 1$ depending only on $n$ and $\eps$ such that
$$
\left|
\log G(x)
-
\log\left(
\fint_{\RR_i} G\,dV
\right)
\right|
\le C_{n, \eps}
\qquad
\text{for all }x\in \RR_i\setminus E.
$$
\end{proposition}
Here and below, $A\asymp_{a, b}B$ means that there is a constant $C\ge1$, depending only on $a$ and $b$, such that $C^{-1}B\le A\le CB$. For a measurable set $X$ and  a function $f\in L^1(X)$, we denote
$$
\fint_X f\,dx:=\frac{1}{\VV(X)}\int_X f\,dx,
$$
and $M_f(r):=\sup_{B_r}|f|$.

The next lemma provides the link between the analytic small cube estimate and the local positivity set estimate. It shows that a global growth bound of order $T$ forces a uniform local $L^2$-doubling bound at scale $T^{-1}$ for all centers outside a set of arbitrarily small measure.
\begin{lemma}\label{sv:lem:fixed-scale}
Fix $n\ge2$, $A_0\ge1$, $L_0\ge1$, and $\delta>0$. There exist constants $\gamma_0\ge1$ and $T_0\ge1$, depending only on $n,A_0,L_0,\delta$, with the following property. For every $T\ge T_0$, suppose that $f$ is a non-zero harmonic function in $B_4$ with $M_f(4)<\infty$ and
\begin{equation*}\label{sv:eq:global-growth}
M_f(4)\le e^{A_0T}M_f\bigl(\frac32\bigr).
\end{equation*}
Then
\begin{equation}\label{sv:eq:bad-centers}
\HH^n\bigl(\bigl\{x\in B_{\frac85}:
\int_{B(x,16r)}f^2\,d y>\gamma_0\int_{B(x,r)}f^2\,d y
\bigr\}\bigr)\le\delta,
\end{equation}
where $r=\frac{L_0}{T}$. Moreover, $T_0$ may be chosen sufficiently large so that all the balls appearing in \eqref{sv:eq:bad-centers} are contained in $B_2$.
\end{lemma}

\begin{proof}
{\it Step 1. Fix the geometric scale and propagate a lower bound.} Normalize $M_f(\frac32)=1$, so 
$1\le M_f(4)\le e^{A_0T}$, and there exists $x_*\in \overline{B_{\frac32}}$ such that $|f(x_*)|=1.$ Let $s=s(n)>0$ be sufficiently small, to be fixed in Step 2. Set $b=\frac{s}{32}$, and for $q\in \overline{B_{\frac85}}$ set $M_{f, q}(r)=\sup_{B(q, r)}|f|$ and $$L_{f,q}(r)=\bigl(\int_{\Sph^{n-1}}f(q+r\omega)^2\,d\mu(\omega)\bigr)^{\frac12},\qquad d\mu=\frac{d\sigma}{|\Sph^{n-1}|}.$$ The translated forms of \eqref{eq:three-sphere} and \eqref{eq:poisson-comparison} give
$$
M_{f, q}(2b)\le C_nL_{f,q}(3b)\le C_nL_{f,q}(b)^\theta L_{f,q}(4b)^{1-\theta}
\le C_n M_{f, q}(b)^\theta M_f(4)^{1-\theta}
$$
for $\theta=\frac{\log(\frac43)}{\log4}$. Equivalently,
\begin{equation}\label{6.7a}
M_{f, q}(b)
\ge
c_n M_{f}(4)^{-\frac{1-\theta}{\theta}}
M_{f, q}(2b)^{\frac1\theta}.
\end{equation}
Fix $p\in \overline{B_{\frac85}}$. We can choose points
$q_0=x_*,q_1,\ldots,q_J=p$
such that
$$
|q_{j+1}-q_j|\le b,
\qquad
J\le J(n).
$$
In particular,
$$
B(q_j,b)\subset B(q_{j+1},2b),
$$
and therefore
$$
M_{f, q_{j+1}}(2b)\ge M_{f, q_j}(b).
$$
Applying \eqref{6.7a} successively along the chain gives
$$
M_{f, q_{j+1}}(b)
\ge
c_n M_f(4)^{-\frac{1-\theta}{\theta}}
M_{f, q_j}(b)^{\frac1\theta}\qquad \text{ for }j\ge 0.
$$
We note $M_{f, q_0}(b)\ge|f(x_*)|=1$. Iterating along the bounded chain gives
\begin{equation}\label{sv:eq:patch-lower-short}
 \sup_{B(p,b)}|f|\ge e^{-C_1T},
\end{equation}
where $C_1=C_1(n,A_0)$. We note \eqref{sv:eq:patch-lower-short} holds for any $p\in \overline{B_{\frac85}}$. Cover $\overline {B_{\frac85}}$ by finitely many cubes
$K_\alpha=p_\alpha+[-\frac s8,\frac s8]^n$, with $p_\alpha\in\overline {B_{\frac85}}$, and choose
$a_\alpha\in\overline{B(p_\alpha,b)}$ such that
$|f(a_\alpha)|\ge e^{-C_1T}$. The number of patches depends only on $n$.

{\it Step 2. Construct holomorphic functions on finitely many patches.} By \cite[Claim C.4]{LPS}, there exists $\rho_n>0$ depending only on $n$ such that, for every $\alpha$, $f$ admits a holomorphic extension to
$$\{z\in \C^n: \dist(z, K_\alpha)<\rho_n\},$$
with
$|f^{\mathbb C}|\le C_n M_f(4)$ there.  Let $k=k(n)$ be the integer in Proposition \ref{prop:C1}. We choose $s=s(n)>0$ sufficiently small so that
$$
\left(3^k+\frac1{32}\right)s<\rho_n.
$$
Since $a_\alpha\in \overline{B(p_\alpha,b)}$, our choice of $s=s(n)$ ensures that
$a_\alpha+s z$ remains in this complex neighborhood of $K_\alpha$ when $|z|\le 3^k$. We may therefore define
$$
G_\alpha(z):=
\bigl(f^{\C}(a_\alpha+s z)\bigr)^2,
\qquad |z|\le 3^k.
$$
Then $G_\alpha$ is holomorphic in $\{z\in\C^n: |z|\le 3^k\}$ and is real and
nonnegative for real $z\in \QQ_0=[-1,1]^n$. Moreover,
$$
G_\alpha(0)=f(a_\alpha)^2\ge e^{-2C_1T},
$$
while
$$
\max_{|z|\le 2^k}|G_\alpha(z)|
\le C_n^2 M_f(4)^2
\le C_n^2 e^{2A_0T}.
$$
Fix $A=A(n,A_0)\ge1$ large enough such that
$$
\max_{|z|\le 2^k}|G_\alpha(z)|
\le |G_\alpha(0)|e^{AT}, \qquad A>\sqrt{n}s,
$$
for all $T\ge T_0$, increasing $T_0$ if necessary.
Thus $G_\alpha$ satisfies the growth hypothesis of Proposition \ref{prop:C1} with $d=AT$.

{\it Step 3. Apply the small cube estimate on two overlapping grids.} Set $m=\lceil AT\rceil$ and $h=\frac{s}{2m}$. On each patch use the two
cubes
$$
 \RR_{\alpha,0}=p_\alpha+[-\frac s4, \frac s4]^n,
 \qquad \RR_{\alpha,1}=\RR_{\alpha,0}+\tfrac h2(1,\ldots,1),
$$
and partition each into $m^n$ cubes of side $h$. For sufficiently large
$T$, 
$$\widetilde{\RR}_{\alpha, 0}:=\frac{{\RR}_{\alpha, 0}-a_\alpha}{s}\subset\frac{1}{2}\QQ_0, \qquad \widetilde{\RR}_{\alpha, 1}:=\frac{{\RR}_{\alpha, 1}-a_\alpha}{s}\subset\frac{1}{2}\QQ_0,$$
and the corresponding small cubes in the rescaled variables have side
length
$$
\frac hs=\frac1{2m}\asymp \frac1{AT}.
$$
Thus we can apply Proposition \ref{prop:C1} to $G_\alpha(z)$ in both $\widetilde{\RR}_{\alpha, 0}$ and $\widetilde{\RR}_{\alpha, 1}$. Pulling the resulting estimates back to the original variables and
taking the union of the exceptional sets for the two grids, we obtain,
for every $\eps>0$, there exists a set $E_\alpha$ with
\begin{equation}\label{exceptionvol}
|E_\alpha|\le C_n\eps
\end{equation}
and a constant $H_{n, \eps}=H(n,\eps)\ge1$ such that, for every
small cube $\QQ$ in either grid,
\begin{equation}\label{sv:eq:avg-short}
H_{n, \eps}^{-1}\fint_\QQ f^2\,dx
\le f(x)^2
\le H_{n, \eps} \fint_\QQ f^2\,dx,
\qquad
\text{ for all }x\in \QQ\setminus E_\alpha.
\end{equation}
Notice that the mesh is fixed before $\eps$ is chosen. We now set $r=\frac{L_0}{T}.$
Note that
\begin{equation*}\label{sv:eq:mesh-short}
h=\frac{s}{2\lceil AT\rceil}
\asymp_{n,A_0}T^{-1}.
\end{equation*}
By the choice of $A$ and $L_0\ge1$, we get 
\begin{equation}\label{diamQ}
\diam(\QQ)=\sqrt n\,h<\frac r2, \qquad \frac{r}{h}\le C(n, A_0, L_0).
\end{equation}
Finally, after increasing $T_0$ if necessary, we may assume that
$$
x+[-21r,21r]^n
\subset \RR_{\alpha,0}\cap \RR_{\alpha,1}
\qquad
\text{for every }x\in K_\alpha.
$$
Indeed, $K_\alpha$ has a fixed positive margin inside
$\RR_{\alpha,0}$, while the displacement $\frac h2$ of the second grid and
the scale $r$ both tend to zero as $T\to\infty$.

{\it Step 4. Remove centers close to bad cubes.} Set $\tau=2^{-n-2}$. We call $\QQ$ a bad cube if
$$\HH^{n}(\QQ\cap E_\alpha)>\tau\HH^n(\QQ).$$ By \eqref{exceptionvol}, the total volume of the bad cubes is
$O_n(\varepsilon)$. Exclude also every center $x$ for which a bad cube
meets $x+[-20r,20r]^n$. Since $\frac rh$ is bounded in terms of
$n,A_0,L_0$, the total measure of these excluded centers is at most
$C(n,A_0,L_0)\varepsilon$. Choosing $\varepsilon$ sufficiently small and
summing over the finitely many patches makes this measure at most $\delta$.
For every remaining center $x$, by the definition of the excluded set, every small cube of either grid meeting $x+[-20r,20r]^n$ is good. 

{\it Step 5. Compare averages on adjacent cubes.} Let $\QQ$ and $\PP$ be two face-adjacent small cubes in the subdivision of $\RR_{\alpha,0}$ such that
$$
\QQ\cup \PP\subset x+[-19r,19r]^n.
$$
There exists a small cube $\VC$ in the half-mesh shifted subdivision of
$\RR_{\alpha,1}$ which overlaps both $\QQ$ and $\PP$, with
$$\HH^n(\VC\cap \QQ)=\HH^n(\VC\cap \PP)=2^{-n}\HH^n(\QQ).$$
Moreover, $\VC\subset\bigl(\QQ\cup \PP\bigr)+[-\frac h2,\frac h2]^n$.
By \eqref{diamQ}, we have
$$
\VC\subset x+[-20r,20r]^n.
$$
Hence $\QQ$, $\PP$, and $\VC$ are all good cubes, and
$$\HH^n(\VC\cap \QQ)>\tau\HH^n(\QQ)\ge\HH^n(\QQ\cap E_\alpha),\qquad\HH^n(\VC\cap \PP)>\tau\HH^n(\PP)\ge\HH^n(\PP\cap E_\alpha).$$
So each overlap contains a point outside
$E_\alpha$, and applying \eqref{sv:eq:avg-short} there gives
\begin{equation}\label{ajcompare}
 H_{n,\eps}^{-4}\fint_\PP f^2\,dx
 \le \fint_\QQ f^2\,dx
 \le H_{n, \eps}^{4}\fint_\PP f^2\,dx.
\end{equation}

{\it Step 6. Pass from cubes to the required balls.} Now let $\QQ_x$ be the small cube of the first grid containing $x$.
If a small cube $\QQ$ of the first grid meets $B(x,16r)$, then, by \eqref{diamQ}
we have
$\QQ\subset x+[-17r,17r]^n.$
Hence $\QQ$ can be joined to $\QQ_x$ by a face-adjacent chain
$$\QQ_x=\QQ^0,\QQ^1,\ldots,\QQ^J=\QQ$$
of small cubes of the first grid, all contained in $x+[-17r,17r]^n$. In particular, every adjacent pair in the chain
satisfies the hypothesis of the preceding comparison estimate, and
all cubes involved are good. Moreover,
$J\le C(n,A_0,L_0).$ Iterating the preceding inequality \eqref{ajcompare} along the chain gives
$$
\fint_\QQ f^2\,dx
\le
H_{n,\eps}^{\,4C(n,A_0,L_0)}
\fint_{\QQ_x} f^2\,dx.
$$
Since all small cubes have the same volume and at most
$C(n,A_0,L_0)$ of them meet $B(x,16r)$, summing over these cubes yields
$$
\int_{B(x,16r)} f^2\,dy
\le
\gamma_0\int_{\QQ_x}f^2\,dy,
$$
where $\gamma_0=\gamma_0(n,A_0,L_0,\delta)$ is independent of $T$. Finally, since $\QQ_x\subset B(x,r),$ we have
$$
\int_{B(x,16r)} f^2\,dy
\le
\gamma_0\int_{B(x,r)}f^2\,dy.
$$
After increasing $T_0$ if necessary, we also have $B(x,16r)\subset B_2$ for every $x\in B_{\frac85}.$
This proves \eqref{sv:eq:bad-centers}.
\end{proof}
Next, we prove the lower bound for the volume of the positivity set under the stable growth assumption.
\begin{proposition}\label{sv:prop:stable-sign}
For every $n\ge3$ and $S\ge1$, there are $c_{n,S}>0$ and $T_{1}=T_1(n,S)$ with the following property.
If $f$ is a non-zero harmonic function in $B_4$, which has $S$-stable
growth as in \eqref{eq:stable}, and
$$
\NN_f\bigl(B_{\frac12}\bigr)\ge T_{1},
$$
then
\begin{equation}\label{sv:eq:stable-sign}
\HH^n\bigl(\{f>0\}\cap B_2\bigr)\ge c_{n,S}.
\end{equation}
No vanishing condition at the center is required.
\end{proposition}

\begin{proof}
As in the proof of Proposition \ref{prop:gradient-upper}, in $B_{\frac85}$ set $g={|\nabla f|}{|f|}^{-1}$ off the zero set, extended by zero
on the zero set. By Theorem \ref{ndupper} and stable growth,
$$
\HH^{n-1}(\{f=0\}\cap B_2)
\le C_n\NN_f(B_2)\le C_nS\NN_f\bigl(B_{\frac12}\bigr).
$$
Consequently, Proposition~\ref{prop:gradient-upper} with $q=\frac34$ and the left hand side inequality of \eqref{eq:stable-comparison} give, for $\NN_f\bigl(B_{\frac12}\bigr)$ sufficiently large,
\begin{equation*}\label{sv:eq:moments}
\int_{B_{\frac85}}\left(\frac{g}{\NN_f\bigl(B_{\frac12}\bigr)}\right)^{\frac12}\,dx\ge c_{0},
\qquad
\int_{B_{\frac85}}\left(\frac{g}{\NN_f\bigl(B_{\frac12}\bigr)}\right)^{\frac34}\,dx\le C_{0},
\end{equation*}
where $c_{0}$ and $C_{0}$ are positive constants depending only on $n$ and $S$.

Choose $a>0$ so that $\sqrt{a}\VV\bigl(B_{\frac 85}\bigr)\le \frac{c_{0}}{2}$, and set
$$E=\left\{x\in B_{\frac 85}:g(x)\ge a\NN_f\bigl(B_{\frac12}\bigr)\right\}.$$ By H\"older's inequality, we have
$$
\frac{c_{0}}{2}\le\int_E\left(\frac{g}{\NN_f\bigl(B_{\frac12}\bigr)}\right)^{\frac12}\,dx
\le C_{0}^{\frac23}(\HH^n(E))^{\frac13},
\qquad \HH^n(E)\ge\eta:=\frac{c_{0}^3}{8C_{0}^2}>0.
$$
By \eqref{eq:gradient-distance}, after fixing $L_0=L_0(n,S)\ge1$ sufficiently large and increasing
the threshold $T_1$, every $x\in E$ has a zero of $f$ in $B(x,r)$,
where $r={L_0}\bigl({\NN_f\bigl(B_{\frac12}\bigr)}\bigr)^{-1}$. Lemma \ref{lem:outer-growth} gives $$M_f(4)\le e^{A_0\NN_f(B_{\frac12})}M_f\bigl(\frac32\bigr)$$ with
$A_0=A_0(n,S)\ge1$. Apply Lemma \ref{sv:lem:fixed-scale} with these
$A_0,L_0$ and $\delta=\frac\eta2$. Removing its exceptional centers from $E$
leaves a set $E'$ such that
\begin{equation*}\label{sv:eq:good-large-set}
\HH^n(E')\ge\frac\eta2,
\qquad
\int_{B(x,16r)}f^2\,dy\le \gamma_0\int_{B(x,r)}f^2\,dy
\qquad\text{for all } x\in E'.
\end{equation*}
For $x\in E'$, choose a zero $z=z(x)\in B(x,r)$. Since $f(z)=0$,
$$
B(x,r)\subset B(z,2r),
\qquad B(z,4r)\subset B(x,16r),
$$
we have $$\int_{B(z,4r)}f^2\,dx\le \gamma_0\int_{B(z,2r)}f^2\,dx.$$
Lemma \ref{sv:lem:balance} therefore yields
\begin{equation}\label{sv:eq:local-sign}
\HH^n\bigl(\{ f>0\}\cap B(z,2r)\bigr)\ge c_{n,S}r^n.
\end{equation}
Let $$\mathcal{F}=\{B(z(x), 2r): x\in E'\}.$$
Choose a maximal pairwise disjoint subfamily $B(z_i,2r)$,
$1\le i\le J$, of $\mathcal{F}$. Maximality and the choice of $z$ give
$E'\subset\bigcup_iB(z_i,5r)$, and hence
\begin{equation}\label{Jnumber}
Jr^n\ge\frac{\HH^n(E')}{5^n\omega_n}\ge c_{n,S}>0.
\end{equation}
For sufficiently large $\NN_f\bigl(B_{\frac12}\bigr)$, all selected balls lie in $B_2$.
Summing \eqref{sv:eq:local-sign} over the disjoint balls, and using \eqref{Jnumber}, we get \eqref{sv:eq:stable-sign}.
\end{proof}

Now we are ready to prove Theorem \ref{sv:thm:main}.
\begin{proof}[Proof of Theorem~\ref{sv:thm:main}]
Let $f(y)=u(\frac y4)$. Recall $\NN=\NN_u\bigl(B_{\frac12}\bigr)$. It suffices to consider the case $\NN<\infty$. Then $f$ is harmonic in $B_4$, $f(0)=0$, and
$$
\NN=\log_2\frac{M_f(4)}{M_f(2)}.
$$
For the function $f$, let $L_f$ and $\beta_f$ denote the spherical $L^2$-mean and frequency function introduced in Section \ref{sec:stable}.
Equations \eqref{eq:frequency-monotonicity}, \eqref{eq:frequency-ratio}, and \eqref{eq:poisson-comparison} imply
$$
\beta_f\bigl(\frac{19}{10}\bigr)\log\bigl(\frac{7}{5}\bigr)
\le\log\frac{L_f\bigl(\frac{7}{2}\bigr)}{L_f\bigl(\frac{5}{2}\bigr)}
\le\log\frac{C_nM_f(4)}{M_f(2)}.
$$
In particular,
\begin{equation}\label{sv:eq:frequency-upper}
\beta_f(\frac{19}{10})\le C_n(\NN+1).
\end{equation}
Let $L_n>0$, $S_n\ge1$, and $c_n>0$ be the dimensional constants
in Lemma \ref{lem:stable-ball}, and let $T_1=T_1(n, S_n)$ be the constant in Proposition \ref{sv:prop:stable-sign}.  Fix the dimensional constant $Q\ge1$ sufficiently large so that
$$
eQ\ge L_n,
\qquad
50c_nQ\ge T_1.
$$
The first condition will ensure that every retained frequency layer
satisfies the scale hypothesis of Lemma \ref{lem:stable-ball}, while the second will ensure that the stable ball produced by that lemma has sufficiently
large inner doubling index to apply Proposition \ref{sv:prop:stable-sign}. Fix
$$
v_*\ge\max\left\{1, \frac{10^4Q}{(1-\sqrt{10}^{-1})^2}\right\},\qquad
v_0:=\beta_f\bigl(\frac{11}{10}\bigr).
$$
We distinguish the cases $v_0\le v_*$ and $v_0>v_*$.

{\it Case 1: $v_0\le v_*$.} By the monotonicity of $\beta_f$ and \eqref{eq:frequency-ratio},
$$
\log\frac{L_f(2t)}{L_f(t)}
=
\int_t^{2t}\frac{\beta_f(s)}{s}\,ds
\le v_*\log 2,
\qquad 0<t\le \frac12,
$$
and therefore
$$
L_f(2t)\le 2^{v_*}L_f(t).
$$
Using polar coordinates and the definition of $L_f$, we obtain
$$
\begin{aligned}
\int_{B_1} f^2\,dy
&=
|\Sph^{n-1}|
\int_0^1 \rho^{n-1}L_f(\rho)^2\,d\rho=
|\Sph^{n-1}|\,2^n
\int_0^{\frac12} t^{n-1}L_f(2t)^2\,dt \\
&\le
2^{n+2v_*}
|\Sph^{n-1}|
\int_0^{\frac12} t^{n-1}L_f(t)^2\,dt=
2^{n+2v_*}
\int_{B_{\frac12}} f^2\,dy.
\end{aligned}
$$
Since $f(0)=0$, Lemma \ref{sv:lem:balance} applied with
$B=B_{\frac12},$ and $d_0=2^{n+2v_*},$
gives
\begin{equation}\label{smallv0}
\HH^n\bigl(\{y\in B_{\frac12}: f(y)>0\}\bigr)
\ge c_n>0.
\end{equation}
By $u(0)=0$ and \cite[Claim A.4]{LPS}, the doubling index $\NN$ is bounded below by a positive dimensional constant. Hence, scaling back to $u$, \eqref{smallv0} yields the desired logarithmic lower bound in Theorem \ref{sv:thm:main}.

{\it Case 2: $v_0>v_*$.} We use the frequency layer construction for $f$ as in Steps 1--3 of the proof of Theorem \ref{thm:main-2}, specifically \eqref{eq:frequency-layers}--\eqref{eq:rescaled-indices}, with the above choice of $Q$. 
For $j=0,1,\ldots$, set
$$
v_j=10^jv_0,
\qquad
I_j=
\left\{
t\in\bigl[\frac{11}{10},\frac{19}{10}\bigr]:
v_j\le\beta_f(t)<10v_j
\right\}.
$$
Discard all intervals of zero length, and retain the original indices. Let $m$ be the number of remaining positive length layers. If $|I_j|>0$, then $v_j\le\beta_f\bigl(\frac{19}{10}\bigr)$. Thus \eqref{sv:eq:frequency-upper} bounds the number
$m$ of positive length layers by
\begin{equation}\label{sv:eq:layer-count}
m\le1+\log_{10}\frac{C_n(1+\NN)}{v_0}
\le C_n\log(1+\NN).
\end{equation}

After discarding the zero length intervals, let $\rho_j$ be the midpoint of
$I_j$ and set
$$
w_j=\frac{|I_j|}{40},
\qquad
\mathcal{S}_j=
\left\{
x:
\rho_j-10w_j\le |x|\le\rho_j+10w_j
\right\}.
$$
As in the proof of Theorem \ref{thm:main-2},
\begin{equation}\label{eq:layer-properties-pos}
0<w_j\le\frac1{50},
\qquad
\sum_jw_j=\frac1{50},
\end{equation}
and
$$
v_j\le\beta_f(t)\le10v_j,
\qquad
t\in[\rho_j-10w_j,\rho_j+10w_j].
$$
The shells $\mathcal{S}_j$ are pairwise separated and contained in $B_2\setminus B_1$.

Retain the indices
$$
\mathcal J_f=\{j:v_jw_j^2\ge Q\}.
$$
For $j\notin\mathcal J_f$,
$$
w_j
<
\sqrt{\frac{Q}{v_j}}
=
\sqrt{\frac{Q}{v_0}}\,10^{-\frac j2}.
$$
Hence, by the choice of $v_*$,
$$
\sum_{j\notin\mathcal J_f}w_j
\le
\sqrt{\frac{Q}{v_0}}
\sum_{j=0}^\infty10^{-\frac j2}
\le\frac1{100}.
$$
Together with \eqref{eq:layer-properties-pos}, this gives
\begin{equation}\label{eq:retained-width-pos}
\sum_{j\in\mathcal J_f}w_j\ge\frac1{100}.
\end{equation}

For every retained index $j\in\mathcal{J}_f$, we have
$$
v_jw_j^2\ge Q,
\qquad
0<w_j\le\frac1{50},
$$
and hence
$$
v_j w_j\ge \frac{Q}{w_j}\ge 50Q,
$$
while
$$
\frac{v_j w_j}{\log\bigl(\frac1{w_j}\bigr)}
\ge
\frac{Q}{w_j\log\bigl(\frac1{w_j}\bigr)}
\ge eQ
\ge L_n.
$$
These inequalities verify the hypotheses of Lemma \ref{lem:stable-ball}. Therefore, for every retained $j\in\mathcal{J}_f$, there exists a ball $D_j=B(x_j,w_j)$ such that $4D_j\subset \mathcal{S}_j,$ and
\begin{equation}\label{sv:conseoflem51}
\NN_f\bigl(\frac{D_j}{2}\bigr)\ge c_nv_jw_j
\ge50c_nQ
\ge T_1,
\qquad
\NN_f(2D_j)\le S_n\NN_f\bigl(\frac{D_j}{2}\bigr).
\end{equation}
The balls $4D_j$, $j\in\mathcal J_f$, are pairwise disjoint because the shells $\mathcal{S}_j$ are pairwise separated.

Apply Proposition~\ref{sv:prop:stable-sign} to
$\widetilde{f_j}(y)=f(x_j+w_jy)$. The hypotheses follow from \eqref{sv:conseoflem51} and $\overline{4D_j}\subset B_2$.
After scaling back, 
$$
\HH^n\bigl(\{f>0\}\cap2D_j\big)\ge c_nw_j^n,\qquad j\in\mathcal J_f.
$$
Disjointness, H\"older's inequality,  \eqref{sv:eq:layer-count}, \eqref{eq:retained-width-pos} and \eqref{sv:conseoflem51} now yield
\begin{align}
\HH^n\bigl(\{f>0\}\cap B_2\big)
&\ge c_n\sum_{j\in\mathcal J_f}w_j^n\notag\ge c_n (\#\mathcal J_f)^{1-n}
 \left(\sum_{j\in\mathcal J_f}w_j\right)^n\\
&\ge c_n m^{1-n}
 \left(\sum_{j\in\mathcal J_f}w_j\right)^n\ge{c_n}{\bigl(\log(1+\NN)\bigr)^{1-n}}.\label{sv:eq:holder-volume}
\end{align}
Finally, under $x=\frac y4$, the ball $B_2$ maps onto $B_{\frac12}$ and the volume is multiplied by $4^{-n}$. This proves \eqref{sv:eq:main}.
\end{proof}

In contrast to \eqref{eq:local-layer-volume}, the contribution of $D_j$ to the volume of the positivity set is of order $w_j^n$, without the
additional frequency factor $v_j$. Consequently, the unweighted
summation in \eqref{sv:eq:holder-volume} retains the factor
$(\#\mathcal J_f)^{1-n}$. Since the number of retained layers is bounded
by $C_n\log(1+\mathcal N)$, this produces the logarithmic factor in
\eqref{sv:eq:main}.

We conclude this section by constructing examples showing that the logarithmic order in Theorem \ref{sv:thm:main} is sharp, up to dimensional constants. The construction below starts from the entire function used by Nazarov, Polterovich and
Sodin in \cite[Section 6.1]{NazarovPolterovichSodin2005}. We adapt their example to higher dimensions by spherical averaging, which produces a harmonic function in $\mathbb R^n$. A contour deformation argument then controls the resulting average outside a fixed cylinder. This higher-dimensional
step is essential for obtaining the exponent $1-n$. We note that simply extending the planar example independently of the remaining variables would yield only the exponent $-1$.

\begin{proposition}\label{sv:prop:sharpness}
For every $n\ge3$, there are a constant $C_n>0$ and non-zero harmonic functions $u_R$ in $\R^n$, indexed by sufficiently large real numbers $R$, such that $u_R(0)=0$,
$$\NN_R:=\NN_{u_R}\bigl(B_{\frac12}\bigr)\longrightarrow\infty,$$
and
\begin{equation}\label{sv:eq:sharpness-upper}
 \HH^n\bigl(\{u_R>0\}\cap B_{\frac12}\bigr)
 \le {C_n}{\bigl(\log(1+\NN_R)\bigr)^{1-n}}.
\end{equation}
\end{proposition}

\begin{proof}
\emph{Step 1. An entire function concentrated in a half strip.}
Set
$$
 \Pi_+=\{z\in\C:\re z\ge0,\ |\im z|\le\frac\pi2\}.
$$
The construction in \cite[Section 6.1, equations (6.3)--(6.4)]{NazarovPolterovichSodin2005}
gives an entire function $\EE$ and an absolute constant $C_0$ such that
\begin{equation}\label{sv:eq:sharpness-entire-input}
 \begin{aligned}
 |\EE(z)|&\le C_0 &&(z\notin\Pi_+),\\
 |\EE(z)-e^{e^z}|&\le C_0 &&(z\in\Pi_+).
 \end{aligned}
\end{equation}
Replacing $\EE(z)$ by
$\tfrac12\bigl(\EE(z)+\overline{\EE(\bar z)}\bigr)$ preserves these estimates and ensures $\EE(\bar z)=\overline{\EE(z)}$. In particular, $\EE$ is real on the real axis. By \eqref{sv:eq:sharpness-entire-input} we have
\begin{equation}\label{sv:eq:sharpness-entire-growth}
 \begin{gathered}
\EE(t)=e^{e^t}+O(1)\quad(t\longrightarrow+\infty),\\
 |\EE(t+is)|\le e^{e^t}+C_0\quad(t,s\in\R).
 \end{gathered}
\end{equation}
For the second inequality, use
$|e^{e^{t+is}}|=e^{e^t\cos s}\le e^{e^t}$ inside $\Pi_+$, and the first bound in \eqref{sv:eq:sharpness-entire-input} outside $\Pi_+$.

\smallskip
\emph{Step 2. A real-valued harmonic average in $\R^n$.} Write $(t,y)\in\R\times\R^{n-1}$, and let $d\mu$ be normalized surface measure on $\Sph^{n-2}\subset\R^{n-1}$. Define
\begin{equation*}\label{sv:eq:sharpness-harmonic-average}
 H_n(t,y)=\int_{\Sph^{n-2}}\EE(t+i\,y\cdot\omega)\dd\mu(\omega).
\end{equation*}
The change of variables $\omega\mapsto-\omega$ and the symmetry of $\EE$ show that $H_n$ is real-valued. Differentiation under the integral is
valid on every compact set, and, for each $\omega\in\Sph^{n-2}$,
$$
 \Delta_{t,y}\EE(t+i\,y\cdot\omega)
 =\left(1-\sum_{j=1}^{n-1}\omega_j^2\right)
 \EE''(t+i\,y\cdot\omega)=0.
$$
Thus $H_n$ is harmonic in $\R^n$. Moreover,
\begin{equation}\label{sv:eq:sharpness-average-growth}
 H_n(t,0)=\EE(t),\qquad
 |H_n(t,y)|\le e^{e^t}+C_0\quad(t\in\R,\ y\in\R^{n-1}).
\end{equation}

\smallskip
\emph{Step 3. A polynomial bound outside a fixed cylinder.}
We claim that
\begin{equation}\label{sv:eq:sharpness-cylinder-bound}
 |H_n(t,y)|\le C_n(1+t+|y|)^{n-2}
 \qquad(t\ge0,\ |y|\ge2).
\end{equation}
By rotational invariance, $H_n(t,y)$ depends on $y$ only through $r=|y|$. Set $$\alpha_n=\bigl(\int_{-1}^1(1-s^2)^{\frac{n-4}{2}}\,d s\bigr)^{-1}.$$
The spherical integration formula gives
\begin{align*}
 H_n(t,y)
 &=\alpha_n\int_{-1}^1 \EE(t+irs)(1-s^2)^{\frac{n-4}{2}}\,d s\notag\\
 &=\frac{\alpha_n}{i\,r^{n-3}}
   \int_{t-ir}^{t+ir}\EE(z)\bigl(r^2+(z-t)^2\bigr)^{\frac{n-4}{2}}\dd z.
 \label{sv:eq:sharpness-vertical-integral}
\end{align*}
On the open vertical segment the power is the positive real power.

Consider the rectangle with vertices $t\pm ir$ and $-1\pm ir$. In its interior,
$$
 \re\bigl(r^2+(z-t)^2\bigr)
 =r^2+(\re z-t)^2-(\im z)^2>0.
$$
Hence, in the interior of the rectangle, we may define
$$
\bigl(r^2+(z-t)^2\bigr)^{\frac{n-4}{2}}
=
e^{
{\frac{n-4}{2}}\mathrm{Log}\bigl(r^2+(z-t)^2\bigr)},
$$
using the holomorphic logarithm on the right half-plane. Since $r^2+(z-t)^2$ does not vanish on any open edge, this branch extends holomorphically across each open edge. On the open right edge
$z=t+is$, $|s|<r$, it agrees with the positive real value $(r^2-s^2)^{\frac{n-4}{2}}.$ Let $\Gamma_{t,r}$ be the path
$$
 t-ir\ \longrightarrow\ -1-ir\ \longrightarrow\ -1+ir
 \ \longrightarrow\ t+ir.
$$
Cauchy's theorem gives
\begin{equation*}\label{sv:eq:sharpness-deformed-integral}
 H_n(t,y)=\frac{\alpha_n}{i\,r^{n-3}}
 \int_{\Gamma_{t,r}}\EE(z)\bigl(r^2+(z-t)^2\bigr)^{\frac{n-4}{2}}\,d z.
\end{equation*}
The only zeros of the expression inside the power
are the right vertices $t\pm ir$, and both are simple. One may first
remove small disks of radius $\eps$ about these vertices and
then let $\eps\downarrow0$. The added arc integrals are
$O_{t,r}(\eps^{\frac{n-4}{2}+1})$ and tend to zero because
$\frac{n-4}{2}\ge-\frac12>-1$. This justifies the deformation also when $n=3$;
the limit is taken for fixed $t,r$ before the estimates below.

All of $\Gamma_{t,r}$ lies outside $\Pi_+$, since $r\ge2>\frac\pi 2$ and the left edge has real part $-1$. Therefore $|\EE|\le C_0$ on this path.
Put $\ell=t+1\ge1$. When $n\ge4$, one has $\frac{n-4}{2}\ge0$, the path length is $2(\ell+r)$, and
$$
 |r^2+(z-t)^2|\le \ell^2+2r^2\le2(\ell+r)^2
 \qquad(z\in\Gamma_{t,r}).$$
It follows that
$$
 |H_n(t,y)|\le C_n\frac{(\ell+r)^{n-3}}{r^{n-3}}
 \le C_n(1+t+r)^{n-2}.
$$
When $n=3$, the power is $-\frac12$. In what follows, let $C$ be an absolute constant that may vary from line to line. Parametrize either horizontal edge by
$z=t-s\pm ir$, $0\le s\le \ell$. Then
$$
 |r^2+(z-t)^2|=|s^2\mp2irs|\ge2rs,
$$
so each horizontal edge contributes at most
$C\int_0^\ell(rs)^{-\frac12}\dd s\le C\sqrt{\frac\ell r}$. On the left edge $z=-1+i\tau$, $|\tau|\le r$, one has
$$
 \re\bigl(r^2+(z-t)^2\bigr)=r^2+\ell^2-\tau^2\ge \ell^2.
$$
That edge contributes at most $\frac{Cr}{\ell}$. As $r^{n-3}=1$ in this case,
$$
 |H_3(t,y)|\le C\bigl(\sqrt{\frac\ell r}+\frac r\ell\bigr)
 \le C(\ell+r).
$$
This proves \eqref{sv:eq:sharpness-cylinder-bound} in every dimension
$n\ge3$.

\smallskip
\emph{Step 4. The positive set and the actual doubling index.}
For $R$ sufficiently large and $x=(x_1,x')\in\R\times\R^{n-1}$, set
\begin{equation*}\label{sv:eq:sharpness-example}
 u_R(x)=H_n(R+2Rx_1,2Rx')-\EE(R).
\end{equation*}
This function is real-valued and harmonic in $\R^n$, and $u_R(0)=0$.
It is non-zero, since its restriction to the $x_1$-axis is
$\EE(R+2Rx_1)-\EE(R)$ and $\EE$ is nonconstant.
For $x\in B_{\frac12}$ one has
$$
 0<R+2Rx_1<2R,\qquad |2Rx'|<R.
$$
If $|x'|\ge R^{-1}$, \eqref{sv:eq:sharpness-cylinder-bound} gives
$$
 |H_n(R+2Rx_1,2Rx')|\le C_n(1+3R)^{n-2}.
$$
In contrast, \eqref{sv:eq:sharpness-entire-growth} implies
$\EE(R)=e^{e^R}+O(1)$. Thus, for all sufficiently large $R$,
\begin{equation*}\label{sv:eq:sharpness-positive-cylinder}
 \{x\in B_{\frac12}:u_R(x)>0\}
 \subset\{x\in B_{\frac12}:|x'|<R^{-1}\}.
\end{equation*}
Then, we obtain
\begin{equation}\label{sv:eq:sharpness-positive-volume}
 \HH^n\bigl(\{x\in B_{\frac12}:u_R(x)>0\}\bigr)
 \le\omega_{n-1}R^{1-n}.
\end{equation}

It remains to relate $R$ to the doubling index, rather than regard $R$
as a growth parameter without checking this relation. For
$\rho\in\{\frac12,1\}$, put $M_R(\rho)=\sup_{B_\rho}|u_R|$.
The upper bound in \eqref{sv:eq:sharpness-average-growth} gives
$$
 M_R(\rho)\le e^{e^{(1+2\rho)R}}+C_0+|\EE(R)|.
$$
By continuity, approaching the boundary point $(\rho,0)$ along the
positive axis also gives
$$
 M_R(\rho)\ge \EE((1+2\rho)R)-\EE(R).
$$
Consequently, for $\rho\in\{\frac 12,1\}$,
$$
 M_R(\rho)=e^{e^{(1+2\rho)R}}+O(e^{e^R}+1),
$$
and hence
\begin{equation*}\label{sv:eq:sharpness-suprema}
 \log M_R(\frac12)=e^{2R}+o(1),\qquad
 \log M_R(1)=e^{3R}+o(1).
\end{equation*}
In particular,
\begin{equation}\label{sv:eq:sharpness-doubling}
 \NN_R=\frac{e^{3R}-e^{2R}+o(1)}{\log2},\qquad
 \log(1+\NN_R)=3R+O(1).
\end{equation}
Combining \eqref{sv:eq:sharpness-positive-volume} and \eqref{sv:eq:sharpness-doubling} proves
\eqref{sv:eq:sharpness-upper}.
\end{proof}

\begin{remark}\label{sharprem}
Theorem \ref{sv:thm:main} and Proposition \ref{sv:prop:sharpness} together show that
the logarithmic order
$\bigl(\log(1+\NN)\bigr)^{1-n}$
is optimal up to dimensional constants. More precisely, no universal lower bound of asymptotically larger order
can hold: if $\Phi(N)$ satisfies
$$
\frac{\Phi(N)}
     {(\log(1+N))^{1-n}}
\longrightarrow\infty\qquad\text{ as } N\to\infty,
$$
then an estimate of the form
$$
\mathcal H^n\bigl(\{u>0\}\cap B_{\frac12}\bigr)
\ge c_n\Phi(\NN)
$$
cannot hold uniformly for all non-zero harmonic functions $u:B_1\to\R$ with $u(0)=0$.
\end{remark}

\appendix
\section{An alternative proof of Nadirashvili's conjecture}\label{sec:appendix}

This appendix develops an observation communicated by Greilhuber, Logunov and Sodin. We give an alternative proof of Nadirashvili's conjecture, recorded in Theorem \ref{logunovlowerbound}. The argument combines the complex analytic lower bound from Section \ref{sec:stable} with a weighted maximum of the frequency and a ball of almost minimal nodal density. It uses neither a frequency layer decomposition nor multiscale induction. The proof does not invoke Theorem \ref{logunovlowerbound}, Proposition  \ref{prop:gradient-upper} and Proposition \ref{prop:stable}, or Theorem \ref{thm:main-2}. 

For simplicity, we state and prove the following formulation on the unit ball $B_1$, from which Theorem \ref{logunovlowerbound} follows by translation and dilation.

\begin{theorem}[Nadirashvili's conjecture]
\label{thm:A1-unit}
For every $n\geq3$ there exists $c_n>0$ such that every non-zero real-valued
harmonic function $u$ in $B_1\subset\R^n$ with $u(0)=0$ satisfies
$$
  \HH^{n-1}\bigl(\{u=0\}\cap B_1\bigr)\geq c_n.
$$
\end{theorem}

For a non-zero real-valued harmonic function $h$ in $B_1$, write
$$
\NN_h=\NN_h\bigl(B_{\frac12}\bigr).
$$
When $h$ is harmonic in $B_4$, we additionally define
$$
 \II(h)=\int_{B_{\frac85}}g_h^{\frac 12}\,dx,
$$
where $g_h=|\nabla h||h|^{-1}$ outside $\{h=0\}$ and $g_h=0$ on $\{h=0\}$.
The integral is initially allowed to be infinite.

We recall the following estimate, which was established without using Theorem \ref{logunovlowerbound}. Corollary \ref{cor:separation} and Lemmas \ref{lem:inner-decay}--\ref{lem:directions} imply that, for each fixed $S\geq1$, there exist constants
$c_{n,S}>0$ and $N_*(n,S)\geq1$, depending only on $n$ and $S$, such that
\begin{equation}\label{eq:A-stable-input}
 \II(h)\geq c_{n,S}\NN_h^{\frac12}
 \quad\text{if}\quad
 \NN_h(B_2)\leq S \NN_h,\qquad \NN_h\geq N_*(n,S).
\end{equation}
We recall the deduction, without invoking
Proposition \ref{prop:stable}.
Normalize $\sup_{B_{\frac32}}|h|=1$.
Let $b>0$ be the decay constant in
Lemma \ref{lem:inner-decay} and apply
Lemma \ref{lem:directions} with $\eps=\frac b8$.
It gives a set $E\subset\mathbb S^{n-1}$ with
$\sigma(E)\geq\delta_{n,S}>0$. For $\xi\in E$, the radial extension $f_\xi$ gives $\phi_\xi(\zeta)=f_\xi(\frac32+\frac\zeta 4)$, holomorphic in
$\mathbb D_2$ and real on $(-2,2)$, with
$$
 \sup_{\mathbb D_2}|\phi_\xi|\leq e^{C_{n,S}\NN_h},\qquad
 |\phi_\xi(0)|\geq e^{-\frac{b\NN_h}{8}},\qquad
 \sup_{[-\frac{7}{100},-\frac{3}{50}]}|\phi_\xi|\leq e^{-b\NN_h}.
$$
Corollary \ref{cor:separation} with $q=\frac 12$ and
$I=(-\frac{2}{25},\frac{2}{25})$ therefore yields
$$
 \int_{\frac{37}{25}}^{\frac{38}{25}}
 \bigl|\partial_r\log|h(r\xi)|\bigr|^{\frac 12}\,dr
 \geq c_{n,S}\NN_h^{\frac 12}
 \qquad(\xi\in E).
$$
Integrating over $E$ proves \eqref{eq:A-stable-input},
since $\bigl|\partial_r\log|h(r\xi)|\bigr|\leq g_h(r\xi)$
away from the zeros. We only use one-variable complex analysis, elementary frequency estimates, spherical harmonics, and a Remez type inequality \cite[Theorem 2]{Erdelyi1992Remez}. This argument does not use a nodal volume lower bound. No condition on $h(0)$ is required.

\begin{lemma}\label{lem:A-unrestricted}
Every non-zero real-valued harmonic function $h$ in $B_4$ satisfies
\begin{equation}\label{eq:A-unrestricted}
 \NN_h\leq C_n\bigl(1+\II(h)^2\bigr).
\end{equation}
\end{lemma}
\begin{proof}
We may assume that $\NN_h$ is sufficiently large and $\II(h)<\infty$.
Let $\beta_h$ be the frequency of $h$ introduced in
Section \ref{sec:stable}.
Choose $t_*\in[\frac{11}{10},\frac{3}{2})$ maximizing
$\beta_h(t)(\frac{3}{2}-t)^{2n}$ on $[\frac{11}{10},\frac{3}{2}]$. Put
$$
 v=\beta_h(t_*),\quad d=\frac{3}{2}-t_*,\quad
 \eta=1-2^{-\frac{1}{2n}},\quad
 w=\frac{\eta d}{20},\quad \rho=t_*+10w.
$$
Lemma \ref{lem:inner-decay} and frequency monotonicity \eqref{eq:frequency-monotonicity} give
$\beta_h(\frac{11}{10})\geq c_n\NN_h$.
Maximality and monotonicity of $\beta_h$ therefore imply
\begin{equation}\label{eq:A-shell}
 vw^{2n}\geq c_n\NN_h,
 \qquad
 v\leq\beta_h(t)\leq2v
 \quad \text{ for } t\in [\rho-10w, \rho+10w].
\end{equation}
Indeed, $vd^{2n}\geq\beta_h(\frac{11}{10})(\frac32-\frac{11}{10})^{2n}$, while on the indicated
interval,
$$
 \beta_h(t)\leq v\left(\frac{d}{\frac{3}{2}-t}\right)^{2n}
 \leq v(1-\eta)^{-2n}=2v.
$$
The resulting shell $\mathcal{S}_{h}:=\{x: \rho-10w\le |x|\le \rho+10w\}$ lies in $B_{\frac 32}\setminus B_{\frac{11}{10}}$,
and $w<1/20$.
Since $w^{2n-1}\log(\frac 1w)$ is bounded above,
$$
 \frac{vw}{\log(\frac 1w)}=\frac{vw^{2n}}{w^{2n-1}\log(\frac 1w)}\geq c_n\NN_h.
$$
Lemma~\ref{lem:stable-ball} supplies a ball $D=B(x_D,w)$
with $\overline{4D}$ in this shell $\mathcal{S}_h$ and
$$
 \NN_h\bigl(\frac D2\bigr)\geq c_nvw,
 \qquad
 \NN_h(2D)\leq S_n\NN_h\bigl(\frac D2\bigr).
$$
The proof of \cite[Lemma 6.5]{LPS} uses the radial growth estimates in \cite[Lemma 6.7]{LPS}, together with a maximum point construction, and does not invoke a nodal volume lower bound. In particular, by \eqref{eq:A-shell} and $0<w<1 $, $\NN_h\bigl(\frac D2\bigr)\geq c_n\NN_h$.
Applying \eqref{eq:A-stable-input} to
$H(y)=h(x_D+wy)$ and using \eqref{eq:A-shell}, we obtain
$$
 \II(h)\geq w^{n-\frac 12}\II(H)\ge c_nw^{n-\frac 12}\NN_h\bigl(\frac D2\bigr)^{\frac12}
 \geq c_n w^{n-\frac 12}(vw)^{\frac 12}
 =c_n(vw^{2n})^{\frac 12}
 \geq c_n\NN_h^{\frac 12}.
$$
This proves \eqref{eq:A-unrestricted}.
\end{proof}

\begin{lemma}\label{lem:A-density-upper}
Let $h$ be a non-zero real-valued harmonic function in $B_4$. Suppose $m>0$ satisfies
\begin{equation}\label{eq:A-density-hyp}
 \HH^{n-1}\bigl(\{h=0\}\cap B(z,t)\bigr)\geq mt^{n-1}
 \quad\text{whenever }h(z)=0\text{ and }B(z,t)\subset B_2.
\end{equation}
Then
\begin{equation}\label{eq:A-density-upper}
 \II(h)^2\leq C_n\bigl(1+m^{-1}\HH^{n-1}\bigl(\{h=0\}\cap B_2\bigr)\bigr).
\end{equation}
\end{lemma}

\begin{proof}
If $\HH^{n-1}(\{h=0\}\cap B_2)=\infty$, the conclusion is immediate. We may therefore assume that $\HH^{n-1}(\{h=0\}\cap B_2)<\infty$.
Set $d(x)=\dist(x,\{h=0\}\cap B_2)$, with $d(x)=\infty$ if the set is empty.
For $x\in B_{\frac 85}\setminus \{h=0\}$, set $R=\min\{\frac1{20},\frac{d(x)}2\}$.
The ball $B(x,R)\subset B_2$ is zero-free, so either $h$ or $-h$
is positive there. The gradient estimate for positive harmonic functions gives
$$
 g_h(x)\leq \frac{C_n}R\leq\frac{C_n}{\min\{\frac1{20},d(x)\}}.
$$
For $0<t<\frac1{20}$, take a maximal $t$-separated family
$\{z_1,\ldots,z_k\}\subset \{h=0\}\cap B_{\frac95}$ with $|z_i-z_j|\ge t$, $i\ne j$.
The balls $B(z_i,\frac t3)$ are disjoint and contained in $B_2$, so
\eqref{eq:A-density-hyp} gives 
$$k\leq C_nm^{-1}t^{1-n}\HH^{n-1}(\{h=0\}\cap B_2).$$
Moreover, $\{x\in B_{\frac 85}:d(x)<t\}$ is covered by the balls $B(z_i,2t)$.
Hence
$$
 |\{x\in B_{\frac 85}:d(x)<t\}|\leq C_nm^{-1}t\HH^{n-1}(\{h=0\}\cap B_2).
$$
Choose $\lambda_0(n)$ large enough that if $g_h(x)>\lambda_0(n)$, then $\min\{\frac1{20},d(x)\}=d(x)$. Consequently, for $\lambda\geq\lambda_0(n)$,
$$
 |\{g_h>\lambda\}\cap B_{\frac85}|
 \leq\min\{\VV(B_{\frac85}),C_nm^{-1}\lambda^{-1}\HH^{n-1}(\{h=0\}\cap B_2)\}.
$$
The layer-cake formula, split at
$\lambda_*\asymp_n1+m^{-1}\HH^{n-1}(\{h=0\}\cap B_2)$ with $\lambda_*\geq\lambda_0(n)$, yields
$$
 \II(h)\leq C_n\lambda_*^{\frac12}
       +C_nm^{-1}\lambda_*^{-\frac12}\HH^{n-1}(\{h=0\}\cap B_2)
 \leq C_n(1+m^{-1}\HH^{n-1}(\{h=0\}\cap B_2))^{\frac12}.
$$
This completes the proof.
\end{proof}

\begin{lemma}\label{lem:A-compactness}
For every finite $\kappa\geq0$ there is $c(n,\kappa)>0$ such that every non-zero real-valued harmonic function $h$ in $B_1$ with $h(0)=0$ and
$$
 \NN_h\leq \kappa
$$
satisfies $\HH^{n-1}(\{h=0\}\cap B_1)\geq c(n,\kappa)$.
Furthermore, for each fixed non-zero real-valued harmonic function $u$ in $B_1$ with $u(0)=0$ and $\HH^{n-1}(\{u=0\}\cap B_1)<\infty$, the number
\begin{equation}\label{eq:A-fixed-density}
 m(u):=\inf_{\substack{u(p)=0,\ r>0\\B(p,r)\subset B_{\frac 12}}}
       r^{1-n}\HH^{n-1}(\{u=0\}\cap B(p,r))
\end{equation}
satisfies $0<m(u)<\infty$. Its positive lower bound at this stage may depend on $u$.
\end{lemma}

\begin{proof}
Suppose that the first assertion fails. Then there is a sequence of non-zero real-valued harmonic functions $h_j$ in $B_1$ such that
$$
 h_j(0)=0,\qquad
 \NN_{h_j}\leq\kappa,\qquad
 \HH^{n-1}(\{h_j=0\}\cap B_1)\longrightarrow0.
$$
Multiplying each $h_j$ by a non-zero constant, we may normalize
$$
 \sup_{B_{\frac{1}{2}}}|h_j|=1,
 \qquad
 \sup_{B_1}|h_j|\leq2^\kappa.
$$
Interior compactness for harmonic functions gives a subsequence converging in $C^\infty_{\mathrm{loc}}(B_1)$ to a harmonic function
$h_\infty$ with
$$
 h_\infty(0)=0,
 \qquad
 \sup_{B_{\frac{1}{2}}}|h_\infty|=1.
$$
In particular, $h_\infty$ is non-zero. The strong maximum principle implies that $h_\infty$ takes both signs in $B_{\frac{1}{2}}$. Choose two closed balls $\overline{B(x_\pm,\delta)}\subset B_{\frac{1}{2}}$ on which $h_\infty$ has opposite strict signs. By uniform convergence, the same signs hold for $h_j$ when $j$ is sufficiently large.

Set
$$
 \nu=\frac{x_+-x_-}{|x_+-x_-|}.
$$
For every $y\in \nu^\perp$ with $|y|<\delta$, the segment joining
$x_-+y$ to $x_++y$ lies in $B_{\frac{1}{2}}$ and meets
$\{h_j=0\}$. Hence the orthogonal projection of
$\{h_j=0\}\cap B_1$ onto $\nu^\perp$ contains a disk of radius
$\delta$. Since orthogonal projection is $1$-Lipschitz,
$$
 \HH^{n-1}(\{h_j=0\}\cap B_1)
 \geq\omega_{n-1}\delta^{n-1}>0,
$$
contradicting the choice of $h_j$.  This is the
elementary projection argument in \cite[Claim 2.3]{LPS}.

For the second assertion, keep $u$ fixed. Let $L_p(s)$ and $\beta_{u, p}(s)$ denote its normalized spherical $L^2$-mean and frequency centered at $p$, respectively. The function
$p\mapsto\beta_{u, p}(\frac{1}{8})$ is continuous and finite on $\overline{{B}_{\frac{1}{2}}}$. Indeed, all relevant spheres lie compactly in $B_1$, and the denominator in the frequency cannot vanish: otherwise $u$ would vanish on the corresponding sphere, hence in its interior by uniqueness of the Dirichlet problem, and then throughout $B_1$ by unique continuation. Consequently,
$$
 \kappa_u:=\sup_{p\in\overline{{B}_{\frac{1}{2}}}}
       \beta_{u, p}\left(\frac{1}{8}\right)<\infty.
$$

Applying the spherical mean comparison
\eqref{eq:poisson-comparison}, the frequency identity \eqref{eq:frequency-ratio}, and monotonicity \eqref{eq:frequency-monotonicity}, all centered at $p$, we have
$$
 \begin{aligned}
 \frac{\sup_{B(p,r)}|u|}
      {\sup_{B(p,\frac{r}{2})}|u|}
 &\leq C_n\frac{L_p(2r)}{L_p(\frac{r}{2})}=C_n\exp\left(
       \int_{\frac{r}{2}}^{2r}
       \frac{\beta_{u, p}(s)}{s}\,ds
       \right)\\
 &\leq C_n4^{\kappa_u},
 \qquad 0<r\leq\frac{1}{16},
 \end{aligned}
$$
where we take $C_n\geq1$. Thus, whenever $u(p)=0$, the rescaled function $y\mapsto u(p+ry)$ has doubling index at most $\log_2 C_n+2\kappa_u$. The first assertion therefore yields
$$
 \HH^{n-1}(\{u=0\}\cap B(p,r))
 \geq c_u r^{n-1},
 \qquad 0<r\leq\frac{1}{16},
$$
where $c_u=c(n,\log_2 C_n+2\kappa_u)>0$ is independent of $p$ and $r$.

We call a ball $B(q, s)$ admissible if $u(q)=0$, $s>0$ and $B(q, s)\subset B_{\frac12}$. If an admissible ball in \eqref{eq:A-fixed-density} has
$r>\frac{1}{16}$, then $r\leq\frac{1}{2}$ and
$B(p,\frac{1}{16})\subset B(p,r)$. Hence
$$
 r^{1-n}\HH^{n-1}(\{u=0\}\cap B(p,r))
 \geq c_u\left(\frac{1}{16r}\right)^{n-1}
 \geq c_u\left(\frac{1}{8}\right)^{n-1}>0.
$$
Together with the previous small scale estimate, this proves $m(u)>0$.
Finally, $B_{\frac{1}{4}}$ is admissible, and
$$
 m(u)
 \leq4^{n-1}
      \HH^{n-1}(\{u=0\}\cap B_{\frac{1}{4}})
 <\infty
$$
by the assumed finiteness of
$\HH^{n-1}(\{u=0\}\cap B_1)$.
\end{proof}

\begin{proof}[Proof of Theorem~\ref{thm:A1-unit}]
We may assume $\HH^{n-1}(\{u=0\}\cap B_1)<\infty$.
By Lemma~\ref{lem:A-compactness}, $m=m(u)>0$.
Choose an admissible ball $B=B(p,r)$ with
$$
 \HH^{n-1}(\{u=0\}\cap B)\leq2mr^{n-1},
$$
and set $h(y)=u(p+\frac{ry}{4})$. Then $h$ is harmonic near $\overline {B_4}$, $h(0)=0$, and
\begin{equation}\label{eq:A-near-min}
 \HH^{n-1}(\{h=0\}\cap B_4)\leq2\cdot4^{n-1}m.
\end{equation}
Every zero-centered ball contained in $B_2$ rescales to an admissible ball in \eqref{eq:A-fixed-density}, so \eqref{eq:A-density-hyp} holds for $h$ with this same $m$. \eqref{eq:A-near-min}, Lemma \ref{lem:A-density-upper} and Lemma \ref{lem:A-unrestricted} therefore yield
$$
 \II(h)^2\leq C_n\left(1+
 \frac{\HH^{n-1}(\{h=0\}\cap B_2)}{m}\right)\leq C_n,
 \qquad \NN_h\leq C_n.
$$
The bounded index assertion of Lemma \ref{lem:A-compactness}, now with a dimensional index bound, implies $\HH^{n-1}(\{h=0\}\cap B_1)\geq c_n$. Together with
\eqref{eq:A-near-min}, this gives $m(u)\geq c_n$.
Finally, since $B_\frac14$ is admissible,
$$
 \HH^{n-1}(\{u=0\}\cap B_1)
 \geq\HH^{n-1}(\{u=0\}\cap B_{\frac14})
 \geq4^{1-n}m(u)\geq c_n.
$$
This completes the proof.
\end{proof}

All results invoked from the preceding sections and from \cite{LPS} are established without using Theorem \ref{logunovlowerbound} or any result whose proof depends on it. Moreover, the argument uses neither a frequency layer decomposition
nor multiscale induction. Instead, it relies on a single weighted frequency selection and an almost minimal density argument.

\bibliography{bib}

@article {LPS,
    AUTHOR = {Logunov, Alexander and Lakshmi Priya, M. E. and Sartori,
              Andrea},
     TITLE = {Almost sharp lower bound for the nodal volume of harmonic
              functions},
   JOURNAL = {Comm. Pure Appl. Math.},
  FJOURNAL = {Communications on Pure and Applied Mathematics},
    VOLUME = {77},
      YEAR = {2024},
    NUMBER = {12},
     PAGES = {4328--4389},
      ISSN = {0010-3640,1097-0312},
   MRCLASS = {31B05},
  MRNUMBER = {4814921},
MRREVIEWER = {Vladimir\ Eiderman},
       DOI = {10.1002/cpa.22207},
       URL = {https://doi.org/10.1002/cpa.22207},
}

@book {ABR,
    AUTHOR = {Axler, Sheldon and Bourdon, Paul and Ramey, Wade},
     TITLE = {Harmonic function theory},
    SERIES = {Graduate Texts in Mathematics},
    VOLUME = {137},
   EDITION = {Second},
 PUBLISHER = {Springer-Verlag, New York},
      YEAR = {2001},
     PAGES = {xii+259},
      ISBN = {0-387-95218-7},
   MRCLASS = {31-01 (30-01 46Exx)},
  MRNUMBER = {1805196},
       DOI = {10.1007/978-1-4757-8137-3},
       URL = {https://doi.org/10.1007/978-1-4757-8137-3},
}

@incollection {Yau1982,
    AUTHOR = {Yau, Shing Tung},
     TITLE = {Problem section},
 BOOKTITLE = {Seminar on {D}ifferential {G}eometry},
    SERIES = {Ann. of Math. Stud., No. 102},
     PAGES = {669--706},
 PUBLISHER = {Princeton Univ. Press, Princeton, NJ},
      YEAR = {1982},
      ISBN = {0-691-08268-5},
   MRCLASS = {53Cxx (58-02)},
  MRNUMBER = {645762},
MRREVIEWER = {Yu.\ Burago},
}

@article {DonnellyFefferman1988,
    AUTHOR = {Donnelly, Harold and Fefferman, Charles},
     TITLE = {Nodal sets of eigenfunctions on {R}iemannian manifolds},
   JOURNAL = {Invent. Math.},
  FJOURNAL = {Inventiones Mathematicae},
    VOLUME = {93},
      YEAR = {1988},
    NUMBER = {1},
     PAGES = {161--183},
      ISSN = {0020-9910,1432-1297},
   MRCLASS = {58G25 (35B60 35P05)},
  MRNUMBER = {943927},
MRREVIEWER = {P.\ G\"{u}nther},
       DOI = {10.1007/BF01393691},
       URL = {https://doi.org/10.1007/BF01393691},
}

@article {Bruning1978,
    AUTHOR = {Br\"{u}ning, Jochen},
     TITLE = {\"{U}ber {K}noten von {E}igenfunktionen des
              {L}aplace-{B}eltrami-{O}perators},
   JOURNAL = {Math. Z.},
  FJOURNAL = {Mathematische Zeitschrift},
    VOLUME = {158},
      YEAR = {1978},
    NUMBER = {1},
     PAGES = {15--21},
      ISSN = {0025-5874,1432-1823},
   MRCLASS = {58G99 (53C20)},
  MRNUMBER = {478247},
MRREVIEWER = {Sh\^{u}kichi\ Tanno},
       DOI = {10.1007/BF01214561},
       URL = {https://doi.org/10.1007/BF01214561},
}

@article {ColdingMinicozzi2011,
    AUTHOR = {Colding, Tobias H. and Minicozzi, II, William P.},
     TITLE = {Lower bounds for nodal sets of eigenfunctions},
   JOURNAL = {Comm. Math. Phys.},
  FJOURNAL = {Communications in Mathematical Physics},
    VOLUME = {306},
      YEAR = {2011},
    NUMBER = {3},
     PAGES = {777--784},
      ISSN = {0010-3616,1432-0916},
   MRCLASS = {58J50 (28A78 35P15 35P20)},
  MRNUMBER = {2825508},
MRREVIEWER = {Julie\ Rowlett},
       DOI = {10.1007/s00220-011-1225-x},
       URL = {https://doi.org/10.1007/s00220-011-1225-x},
}

@article {SoggeZelditch2011,
    AUTHOR = {Sogge, Christopher D. and Zelditch, Steve},
     TITLE = {Lower bounds on the {H}ausdorff measure of nodal sets},
   JOURNAL = {Math. Res. Lett.},
  FJOURNAL = {Mathematical Research Letters},
    VOLUME = {18},
      YEAR = {2011},
    NUMBER = {1},
     PAGES = {25--37},
      ISSN = {1073-2780},
   MRCLASS = {58J50 (28A78 35P15 35R01)},
  MRNUMBER = {2770580},
MRREVIEWER = {Nelia\ Charalambous},
       DOI = {10.4310/MRL.2011.v18.n1.a3},
       URL = {https://doi.org/10.4310/MRL.2011.v18.n1.a3},
}

@article {SoggeZelditch2012,
    AUTHOR = {Sogge, Christopher D. and Zelditch, Steve},
     TITLE = {Lower bounds on the {H}ausdorff measure of nodal sets {II}},
   JOURNAL = {Math. Res. Lett.},
  FJOURNAL = {Mathematical Research Letters},
    VOLUME = {19},
      YEAR = {2012},
    NUMBER = {6},
     PAGES = {1361--1364},
      ISSN = {1073-2780,1945-001X},
   MRCLASS = {58C40 (28A78 35P15 35R01)},
  MRNUMBER = {3091613},
MRREVIEWER = {Nelia\ Charalambous},
       DOI = {10.4310/MRL.2012.v19.n6.a14},
       URL = {https://doi.org/10.4310/MRL.2012.v19.n6.a14},
}

@article {Steinerberger2014,
    AUTHOR = {Steinerberger, Stefan},
     TITLE = {Lower bounds on nodal sets of eigenfunctions via the heat
              flow},
   JOURNAL = {Comm. Partial Differential Equations},
  FJOURNAL = {Communications in Partial Differential Equations},
    VOLUME = {39},
      YEAR = {2014},
    NUMBER = {12},
     PAGES = {2240--2261},
      ISSN = {0360-5302,1532-4133},
   MRCLASS = {58J50 (35J05 35K08 35P15 35R01)},
  MRNUMBER = {3259555},
MRREVIEWER = {Jia-Yong\ Wu},
       DOI = {10.1080/03605302.2014.942739},
       URL = {https://doi.org/10.1080/03605302.2014.942739},
}

@article {Logunov2018,
    AUTHOR = {Logunov, Alexander},
     TITLE = {Nodal sets of {L}aplace eigenfunctions: proof of
              {N}adirashvili's conjecture and of the lower bound in {Y}au's
              conjecture},
   JOURNAL = {Ann. of Math. (2)},
  FJOURNAL = {Annals of Mathematics. Second Series},
    VOLUME = {187},
      YEAR = {2018},
    NUMBER = {1},
     PAGES = {241--262},
      ISSN = {0003-486X,1939-8980},
   MRCLASS = {58J50 (35J05 35P15 35P20 35R01)},
  MRNUMBER = {3739232},
MRREVIEWER = {Leonid\ Friedlander},
       DOI = {10.4007/annals.2018.187.1.5},
       URL = {https://doi.org/10.4007/annals.2018.187.1.5},
}

@article {Dong1992,
    AUTHOR = {Dong, Rui-Tao},
     TITLE = {Nodal sets of eigenfunctions on {R}iemann surfaces},
   JOURNAL = {J. Differential Geom.},
  FJOURNAL = {Journal of Differential Geometry},
    VOLUME = {36},
      YEAR = {1992},
    NUMBER = {2},
     PAGES = {493--506},
      ISSN = {0022-040X,1945-743X},
   MRCLASS = {58G25 (35P99)},
  MRNUMBER = {1180391},
MRREVIEWER = {Stig\ I.\ Andersson},
       URL = {http://projecteuclid.org/euclid.jdg/1214448750},
}

@article {DonnellyFefferman1990,
    AUTHOR = {Donnelly, Harold and Fefferman, Charles},
     TITLE = {Nodal sets for eigenfunctions of the {L}aplacian on surfaces},
   JOURNAL = {J. Amer. Math. Soc.},
  FJOURNAL = {Journal of the American Mathematical Society},
    VOLUME = {3},
      YEAR = {1990},
    NUMBER = {2},
     PAGES = {333--353},
      ISSN = {0894-0347,1088-6834},
   MRCLASS = {58G25 (35P05)},
  MRNUMBER = {1035413},
MRREVIEWER = {H.-B.\ Rademacher},
       DOI = {10.2307/1990956},
       URL = {https://doi.org/10.2307/1990956},
}

@article {HardtSimon1989,
    AUTHOR = {Hardt, Robert and Simon, Leon},
     TITLE = {Nodal sets for solutions of elliptic equations},
   JOURNAL = {J. Differential Geom.},
  FJOURNAL = {Journal of Differential Geometry},
    VOLUME = {30},
      YEAR = {1989},
    NUMBER = {2},
     PAGES = {505--522},
      ISSN = {0022-040X,1945-743X},
   MRCLASS = {58E05 (35J99)},
  MRNUMBER = {1010169},
MRREVIEWER = {Fang\ Hua\ Lin},
       URL = {http://projecteuclid.org/euclid.jdg/1214443599},
}

@article {Nadirashvili1988,
    AUTHOR = {Nadirashvili, Nikolai S.},
     TITLE = {The length of the nodal curve of an eigenfunction of the
              {L}aplace operator},
   JOURNAL = {Uspekhi Mat. Nauk},
  FJOURNAL = {Akademiya Nauk SSSR i Moskovskoe Matematicheskoe Obshchestvo.
              Uspekhi Matematicheskikh Nauk},
    VOLUME = {43},
      YEAR = {1988},
    NUMBER = {4(262)},
     PAGES = {219--220},
      ISSN = {0042-1316},
   MRCLASS = {58G25 (35P05)},
  MRNUMBER = {969585},
MRREVIEWER = {G.\ V.\ Rozenblum},
       DOI = {10.1070/RM1988v043n04ABEH001905},
       URL = {https://doi.org/10.1070/RM1988v043n04ABEH001905},
}

@article {Lin1991,
    AUTHOR = {Lin, Fang-Hua},
     TITLE = {Nodal sets of solutions of elliptic and parabolic equations},
   JOURNAL = {Comm. Pure Appl. Math.},
  FJOURNAL = {Communications on Pure and Applied Mathematics},
    VOLUME = {44},
      YEAR = {1991},
    NUMBER = {3},
     PAGES = {287--308},
      ISSN = {0010-3640,1097-0312},
   MRCLASS = {58G11 (35J05 35K05 58G03)},
  MRNUMBER = {1090434},
MRREVIEWER = {Robert\ McOwen},
       DOI = {10.1002/cpa.3160440303},
       URL = {https://doi.org/10.1002/cpa.3160440303},
}

@article {GarofaloLin1986,
    AUTHOR = {Garofalo, Nicola and Lin, Fang-Hua},
     TITLE = {Monotonicity properties of variational integrals, {$A_p$}
              weights and unique continuation},
   JOURNAL = {Indiana Univ. Math. J.},
  FJOURNAL = {Indiana University Mathematics Journal},
    VOLUME = {35},
      YEAR = {1986},
    NUMBER = {2},
     PAGES = {245--268},
      ISSN = {0022-2518,1943-5258},
   MRCLASS = {35J20 (35J10 42B25)},
  MRNUMBER = {833393},
MRREVIEWER = {Stavros\ A.\ Belbas},
       DOI = {10.1512/iumj.1986.35.35015},
       URL = {https://doi.org/10.1512/iumj.1986.35.35015},
}

@article {Han,
    AUTHOR = {Han, Qing},
     TITLE = {Nodal sets of harmonic functions},
   JOURNAL = {Pure Appl. Math. Q.},
  FJOURNAL = {Pure and Applied Mathematics Quarterly},
    VOLUME = {3},
      YEAR = {2007},
    NUMBER = {3, Special Issue: In honor of Leon Simon. Part 2},
     PAGES = {647--688},
      ISSN = {1558-8599,1558-8602},
   MRCLASS = {31B05},
  MRNUMBER = {2351641},
MRREVIEWER = {Stephen\ J.\ Gardiner},
       DOI = {10.4310/PAMQ.2007.v3.n3.a2},
       URL = {https://doi.org/10.4310/PAMQ.2007.v3.n3.a2},
}

@article{NazarovPolterovichSodin2005,
    AUTHOR = {Nazarov, F\"{e}dor and Polterovich, Leonid and Sodin, Mikhail},
     TITLE = {Sign and area in nodal geometry of {L}aplace eigenfunctions},
   JOURNAL = {Amer. J. Math.},
  FJOURNAL = {American Journal of Mathematics},
    VOLUME = {127},
      YEAR = {2005},
    NUMBER = {4},
     PAGES = {879--910},
      ISSN = {0002-9327,1080-6377},
   MRCLASS = {58J50 (35J25 35P20)},
  MRNUMBER = {2154374},
MRREVIEWER = {Alessandro\ Savo},
       URL =
              {http://muse.jhu.edu/journals/american_journal_of_mathematics/v127/127.4nazarov.pdf},
}

@article {RoyFortin2015,
    AUTHOR = {Roy-Fortin, Guillaume},
     TITLE = {Nodal sets and growth exponents of {L}aplace eigenfunctions on
              surfaces},
   JOURNAL = {Anal. PDE},
  FJOURNAL = {Analysis \& PDE},
    VOLUME = {8},
      YEAR = {2015},
    NUMBER = {1},
     PAGES = {223--255},
      ISSN = {2157-5045,1948-206X},
   MRCLASS = {58J50 (35P05 35R01)},
  MRNUMBER = {3336925},
MRREVIEWER = {Mohammed\ El A\"{\i}di, Universidad Nacional de Colombia},
       DOI = {10.2140/apde.2015.8.223},
       URL = {https://doi.org/10.2140/apde.2015.8.223},
}

@incollection {LogunovMalinnikovaReview2019,
    AUTHOR = {Logunov, Alexander and Malinnikova, Eugenia},
     TITLE = {Review of {Y}au's conjecture on zero sets of {L}aplace
              eigenfunctions},
 BOOKTITLE = {Current developments in mathematics 2018},
     PAGES = {179--212},
 PUBLISHER = {Int. Press, Somerville, MA},
      YEAR = {2020},
      ISBN = {978-1-57146-387-6},
   MRCLASS = {58J50 (35P30)},
  MRNUMBER = {4363378},
}

@article {Kukavica1998,
    AUTHOR = {Kukavica, Igor},
     TITLE = {Quantitative uniqueness for second-order elliptic operators},
   JOURNAL = {Duke Math. J.},
  FJOURNAL = {Duke Mathematical Journal},
    VOLUME = {91},
      YEAR = {1998},
    NUMBER = {2},
     PAGES = {225--240},
      ISSN = {0012-7094,1547-7398},
   MRCLASS = {35J15},
  MRNUMBER = {1600578},
MRREVIEWER = {Philip\ W.\ Schaefer},
       DOI = {10.1215/S0012-7094-98-09111-6},
       URL = {https://doi.org/10.1215/S0012-7094-98-09111-6},
}

@article {HardtEtAl1999,
    AUTHOR = {Hardt, Robert and Hoffmann-Ostenhof, Maria and Hoffmann-Ostenhof, Thomas
              and Nadirashvili, Nikolai S.},
     TITLE = {Critical sets of solutions to elliptic equations},
   JOURNAL = {J. Differential Geom.},
  FJOURNAL = {Journal of Differential Geometry},
    VOLUME = {51},
      YEAR = {1999},
    NUMBER = {2},
     PAGES = {359--373},
      ISSN = {0022-040X,1945-743X},
   MRCLASS = {35J15 (35B65)},
  MRNUMBER = {1728303},
MRREVIEWER = {Rolando\ Magnanini},
       URL = {http://projecteuclid.org/euclid.jdg/1214425070},
}

@article {Mangoubi2013,
    AUTHOR = {Mangoubi, Dan},
     TITLE = {The effect of curvature on convexity properties of harmonic
              functions and eigenfunctions},
   JOURNAL = {J. Lond. Math. Soc. (2)},
  FJOURNAL = {Journal of the London Mathematical Society. Second Series},
    VOLUME = {87},
      YEAR = {2013},
    NUMBER = {3},
     PAGES = {645--662},
      ISSN = {0024-6107,1469-7750},
   MRCLASS = {58J50 (35J15 35P20 35R01 53C21 58E20)},
  MRNUMBER = {3073669},
MRREVIEWER = {Tanya\ J.\ Christiansen},
       DOI = {10.1112/jlms/jds067},
       URL = {https://doi.org/10.1112/jlms/jds067},
}

@article {NaberValtorta2017,
    AUTHOR = {Naber, Aaron and Valtorta, Daniele},
     TITLE = {Volume estimates on the critical sets of solutions to elliptic
              {PDE}s},
   JOURNAL = {Comm. Pure Appl. Math.},
  FJOURNAL = {Communications on Pure and Applied Mathematics},
    VOLUME = {70},
      YEAR = {2017},
    NUMBER = {10},
     PAGES = {1835--1897},
      ISSN = {0010-3640,1097-0312},
   MRCLASS = {35J15 (35A20 35R01)},
  MRNUMBER = {3688031},
MRREVIEWER = {Xinhua\ Ji},
       DOI = {10.1002/cpa.21708},
       URL = {https://doi.org/10.1002/cpa.21708},
}

@article{Nadi,
  author  = {Nadirashvili, Nikolai S.},
  title   = {Geometry of nodal sets and multiplicity of eigenvalues},
  journal = {Current Developments in Mathematics},
  volume  = {1997},
  number  = {1},
  pages   = {231--235},
  year    = {1997},
  doi     = {10.4310/CDM.1997.v1997.n1.a16}
}

@book {Conway,
    AUTHOR = {Conway, John B.},
     TITLE = {Functions of one complex variable},
    SERIES = {Graduate Texts in Mathematics},
    VOLUME = {11},
   EDITION = {Second},
 PUBLISHER = {Springer-Verlag, New York-Berlin},
      YEAR = {1978},
     PAGES = {xiii+317},
      ISBN = {0-387-90328-3},
   MRCLASS = {30-01},
  MRNUMBER = {503901},
MRREVIEWER = {P.\ Lappan},
}

@article {TY,
    AUTHOR = {Tikhonov, Sergey Yu. and Yuditskii, Peter},
     TITLE = {Sharp {R}emez inequality},
   JOURNAL = {Constr. Approx.},
  FJOURNAL = {Constructive Approximation. An International Journal for
              Approximations and Expansions},
    VOLUME = {52},
      YEAR = {2020},
    NUMBER = {2},
     PAGES = {233--246},
      ISSN = {0176-4276,1432-0940},
   MRCLASS = {41A17 (30C35 41A44 41A50)},
  MRNUMBER = {4170300},
MRREVIEWER = {Wies\l aw\ Ple\'{s}niak},
       DOI = {10.1007/s00365-019-09473-2},
       URL = {https://doi.org/10.1007/s00365-019-09473-2},
}

@article {Logunov2018-2,
    AUTHOR = {Logunov, Alexander},
     TITLE = {Nodal sets of {L}aplace eigenfunctions: polynomial upper
              estimates of the {H}ausdorff measure},
   JOURNAL = {Ann. of Math. (2)},
  FJOURNAL = {Annals of Mathematics. Second Series},
    VOLUME = {187},
      YEAR = {2018},
    NUMBER = {1},
     PAGES = {221--239},
      ISSN = {0003-486X,1939-8980},
   MRCLASS = {58J50 (35J05 35P20 35R01)},
  MRNUMBER = {3739231},
MRREVIEWER = {Leonid\ Friedlander},
       DOI = {10.4007/annals.2018.187.1.4},
       URL = {https://doi.org/10.4007/annals.2018.187.1.4},
}

@misc{LiYing-1,
      title={Lower Bound of Nodal Sets in Elliptic Homogenization and Functions with Strong Maximum Principle}, 
      author={Jiahuan Li and Zhichen Ying},
      year={arXiv:2512.12305, 2025},
      eprint={2512.12305},
      archivePrefix={arXiv},
      primaryClass={math.AP},
      url={https://arxiv.org/abs/2512.12305}, 
}

@misc{LiWangYing,
      title={Nadirashvili' Conjecture for Elliptic PDEs and its Applications}, 
      author={Jiahuan Li and Junyuan Wang and Zhichen Ying},
      year={arXiv:2508.07861, 2025},
      eprint={2508.07861},
      archivePrefix={arXiv},
      primaryClass={math.AP},
      url={https://arxiv.org/abs/2508.07861}, 
}

@article{Gelfond1934,
  author  = {Gelfond, Alexander},
  title   = {{\"U}ber die harmonischen {Funktionen}},
  journal = {Trav. Inst. Stekloff},
  volume  = {5},
  year    = {1934},
  pages   = {149--158}
}

@article {Erdelyi1992Remez,
    AUTHOR = {Erd\'{e}lyi, Tam\'{a}s},
     TITLE = {Remez-type inequalities on the size of generalized
              polynomials},
   JOURNAL = {J. London Math. Soc. (2)},
  FJOURNAL = {Journal of the London Mathematical Society. Second Series},
    VOLUME = {45},
      YEAR = {1992},
    NUMBER = {2},
     PAGES = {255--264},
      ISSN = {0024-6107,1469-7750},
   MRCLASS = {41A17},
  MRNUMBER = {1171553},
MRREVIEWER = {Adam\ Piotr\ W\'{o}jcik},
       DOI = {10.1112/jlms/s2-45.2.255},
       URL = {https://doi.org/10.1112/jlms/s2-45.2.255},
}
\bibliographystyle{plain}

\end{document}